\documentclass[11pt,reqno,intlimits]{amsart}

\usepackage[T1]{fontenc}
\usepackage{amsfonts,amssymb,amsmath,amsthm}
\usepackage{graphicx}
\usepackage{colortbl}
\usepackage[colorlinks=true,urlcolor=blue,citecolor=red,linkcolor=blue,linktocpage,pdfpagelabels,bookmarksnumbered,bookmarksopen]{hyperref}

\makeatletter

\newtheorem{theorem}{Theorem}[section]
\newtheorem{proposition}[theorem]{Proposition}
\newtheorem{lemma}[theorem]{Lemma}
\newtheorem{corollary}[theorem]{Corollary}
\theoremstyle{definition}
\newtheorem{definition}[theorem]{Definition}

\theoremstyle{remark}
\newtheorem{remark}[theorem]{Remark}
\numberwithin{equation}{section}

\@ifundefined{mathbb}{%
  \newcommand{\R}{{\rm I\mskip -3.5mu R}}
  \newcommand{\N}{{\rm I\mskip -3.5mu N}}
  \newcommand{\Z}{{\rm I\mskip -3.5mu Z}}
}{%
  \newcommand{\R}{{\mathbb R}}
  \newcommand{\N}{{\mathbb N}}
}
\newcommand{\C}{{\mathcal C}}
\newcommand{\X}{{\mathcal X}}
\newcommand{\abs}[1]{\mathopen|#1\mathclose|}
\newcommand{\bigabs}[1]{\bigl|#1\bigr|}
\newcommand{\biggabs}[1]{\biggl|#1\biggr|}
\@ifundefined{leqslant}{}{\let\le=\leqslant  \let\leq=\leqslant}
\@ifundefined{geqslant}{}{\let\ge=\geqslant  \let\geq=\geqslant}
\@ifundefined{varnothing}{}{}
\newcommand{\Z}{{\mathcal Z}}
\newcommand{\IS}{{\mathbb S}}
\renewcommand{\P}{{\mathcal P}}
\newcommand{\Y}{{\mathcal Y}}

\let\Im@ORIG=\Im
\renewcommand{\Im}{\operatorname{\Im@ORIG m}}
\def\e{\varepsilon}
\def\l{\lambda}
\def\a{\alpha}
\def\b{\beta}
\def\S{\mathcal{S}}
\def\di12{\mathcal{D}^{1,2}(\R^n)}

\newcommand{\norm}[1]{\mathopen\|#1\mathclose\|}
\newcommand{\wto}{\rightharpoonup}
\newcommand{\rad}{{\text{\upshape rad}}}

\newcommand{\intd}{\,\mathrm{d}}
\newcommand{\acto}{\diamond}
\newcommand{\diff}{{\mathrm d}}
\DeclareMathOperator{\spanned}{span}

\makeatother

\title{Entire radial and nonradial solutions for systems with critical growth}
\thanks{The first author is partially supported by GNAMPA. The first two
  authors are supported by PRIN-2012-grant ``Variational and
  perturbative aspects of nonlinear differential problems''.
  The third author is partially supported by the project
  ``Existence and asymptotic behavior of solutions
  to systems of semilinear elliptic partial differential equations''
  (T.1110.14)
  of the \emph{Fonds de la Recherche Fondamentale Collective},
  Belgium. }

\author[Gladiali]{Francesca  Gladiali}
\address{Francesca Gladiali, Dipartimento Polcoming, Universit\`a  di Sassari  - Via Piandanna 4, 07100 Sassari, Italy.}

\email{fgladiali@uniss.it}
\author[Grossi]{Massimo Grossi}
\address{Massimo Grossi, Dipartimento di Matematica, Universit\`a di Roma
La Sapienza, P.le A. Moro 2 - 00185 Roma, Italy.}
\email{massimo.grossi@uniroma1.it}
\author[Troestler]{Christophe Troestler}
\address{Christophe Troestler, D\'epartement de math\'ematique,
  Universit\'e de Mons, place du parc 20, B-7000 Mons, Belgium.}
\email{Christophe.Troestler@umons.ac.be}



\begin{document}

\begin{abstract}
  In this paper we establish existence of radial and nonradial solutions
  to the system
  \begin{equation*}
    \begin{cases}
      \displaystyle
      -\Delta u_1 = F_1(u_1,u_2)
      &\text{in }\R^N,\\
      -\Delta u_2 = F_2(u_1,u_2)
      &\text{in }\R^N,\\
      u_1\geq 0,\ u_2\geq 0 &\text{in }\R^N,\\[1\jot]
      u_1,u_2\in D^{1,2}(\R^N),
    \end{cases}
  \end{equation*}
 where $F_1,F_2$ are  nonlinearities with critical behavior.
\end{abstract}
\maketitle

\section{Introduction}
The aim of this paper is to prove existence of radial and nonradial
solutions to some nonlinear systems
\begin{equation} \nonumber\label{-1}
    \begin{cases}
      \displaystyle
      -\Delta u_1 = F_1(u_1,u_2)
      &\text{in }\R^N,\\
      -\Delta u_2 = F_2(u_1,u_2)
      &\text{in }\R^N,\\
      u_1\geq 0,\ u_2\geq 0  &\text{in }\R^N,\\[1\jot]
      u_1,u_2\in D^{1,2}(\R^N),
    \end{cases}
  \end{equation}
where $F_1,F_2$ are nonlinearities with critical behavior in the Sobolev sense, $N\geq 3$ and $D^{1,2}(\R^N) = \bigl\{u\in L^{2^*}(\R^N)\text{ such that }|\nabla u| \in L^2(\R^N)\bigr\}$  with  $2^*=\frac{2N}{N-2}$.
A common feature of the systems that we will study is their
invariance by \emph{translations} and \emph{dilations}.
Papers on existence or qualitative properties of solutions to systems
with critical growth in $\R^N$ are very few, due to the lack of compactness given by the Talenti bubbles and the difficulties arising for the lack of good variational methods. The first example of system which we consider is given by
\begin{equation} \label{0b}
  \begin{cases}
    \displaystyle
    -\Delta u_1 =\a u_1^{ 2^*-1}  +(1-\a) u_1^\frac2{N-2}u_2^\frac{N}{N-2}
    &\text{in }\R^N,\\
       -\Delta u_2 =\a u_2^{ 2^*-1}+(1-\a) u_2^\frac2{N-2}u_1^\frac{N}{N-2}
    &\text{in }\R^N,\\
    u_1\geq 0,\ u_2\geq 0, \quad   u_1,u_2\in D^{1,2}(\R^N),
  \end{cases}
\end{equation}
where $N\ge 3$ and $\a$ is a real parameter. This system, also known
as Gross-Pitaevskii, arises in many physical contexts such as nonlinear
optics and the Hartree-Fock theory, see \cite{M} for its derivation,
and it is very studied mainly in the
cubic case, which corresponds to the critical case in $\R^4$ or on
bounded domains where the cubic exponent is subcritical in $\R^3$. It
is coupled when $1-\a\neq 0$ and cooperative when
$1-\a>0$. Physically, this condition means the attractive interaction
of the states $u_1$ and $u_2$, while  $1-\a<0$ means the repulsive
interaction between them.
Note that System~\eqref{0b} has a gradient structure with the energy functional
\begin{equation*}
  \mathcal{E}(u_1,u_2)
  = \frac 12 \int_{\R^N} |\nabla u_1|^2+|\nabla u_2|^2
  - \frac{N-2}{2N} \int_{\R^N} \a \bigl(u_1^{2^*}+u_2^{2^*}\bigr)
  + (1-\a) \Bigl(u_1^\frac N{N-2} u_2^\frac N{N-2}\Bigr)
\end{equation*}
even if it is not so easy to apply variational methods to find
solutions. System \eqref{0b} was already considered in \cite{GLW} where the existence of infinitely many nontrivial solutions is obtained using a perturbation argument.\par
Another particular case of~\eqref{-1} is the following generalization of the system considered
by O.~Druet, E.~Hebey~\cite{DH}, namely
\begin{equation}\label{dh}
  \begin{cases}
    -\Delta u_1=\left[\left(\a u_1^{2} + (1-\a) u_2^{2}\right)^2\right]^{\frac1{N-2}}u_1
    & \text{in }\R^N,\\[1\jot]
    -\Delta u_2= \left[ \left((1-\a) u_1^{2} + \a u_2^{2}\right)^2\right]^{\frac 1{N-2}}u_2
    & \text{in }\R^N,\\[1\jot]
    u_1\geq 0,\ u_2\geq 0, \quad  
    u_1,u_2\in D^{1,2}(\R^N).
  \end{cases}
\end{equation}
In \cite{DH} the case of $\a=\frac 12$ was studied and the
stability of solutions on manifolds was considered. Further, the
radial symmetry and uniqueness of the solutions in $\R^N$ is
proved.
 
We also mention the paper \cite{CSW} where the radial symmetry of
solutions is proved for a particular critical nonlinearity.
 
The starting point of our study is the paper \cite{GGT} where we studied the existence of radial solutions for the $k\times k$ system of equations
\begin{equation} \label{0}
  \begin{cases}
    \displaystyle
    -\Delta u_i = \sum\limits_{j=1}^ka_{ij}u_j^{2^* - 1}
    &\text{in }\R^N,\\
    u_i>0 &\text{in }\R^N,\\[1\jot]
    u_i\in D^{1,2}(\R^N),
  \end{cases}
\end{equation}
for $i=1,\dotsc ,k$,
where $N\ge 3$ and the
matrix $A := (a_{ij})_{i,j=1,\dotsc ,k}$ is symmetric and satisfies 
\begin{equation} \label{0a}
\sum\limits_{j=1}^k a_{ij} = 1\text{ for any } i = 1,\dotsc ,k.
\end{equation}
Note that the case $k=2$ and
$A = \bigl(\begin{smallmatrix}0&1\\1&0 \end{smallmatrix}\bigr)$ is
known in the literature as nonlinearity belonging to
the \emph{critical hyperbola}.

Under the assumption  \eqref{0a} it is straightforward that system \eqref{0} always admits the trivial solutions
\begin{equation*}
  u_1 = \dots =u_k = U_{\delta,y}(x)
  := \frac{\left[N(N-2)\delta^2\right]^{\frac{N-2}4}}{
    (\delta^2+|x-y|^2)^{\frac{N-2}2}}
\end{equation*}
for any $\delta>0$ and $y\in \R^N$. 
To simplify the notation, let
\begin{equation}\label{sol}
  U(x) := U_{1,0}(x)
  = \frac {\left[N(N-2)\right]^{\frac{N-2}4}}{
    (1+|x|^2)^{\frac{N-2}2}}.
\end{equation}
A careful study of the linearized system of \eqref{0} at this trivial
solution allows us to prove the existence of nontrivial radial
solutions when the eigenvalues of the matrix $A$ reach some specific
values using bifurcation theory.

Note that System \eqref{0} does not have a
variational structure and indeed our methods do not require it.

Even if the existence of radial solutions to some of the previous
examples~\eqref{0b}--\eqref{0} is a new
result, the main interest is the existence of nonradial ones.
Nonradial solutions may be found mainly for noncooperative systems
where the lack of the maximum principle can give a symmetry breaking
of the solutions. Indeed, in \cite{DH} and \cite{CSW}, the radial
symmetry of the solutions is proved in a particular cooperative case.

In this paper we want to purse several goals.  First, we want to
introduce a new setting which allows us to consider Systems
\eqref{0b}--\eqref{0} jointly.  Indeed all these problems admit the
trivial solutions $ u_1  =u_2 = U_{\delta,y}(x)$ which is the starting
point to apply the bifurcation theory like in~\cite{GGT}.
A general treatment of these problems is possible since we
significantly improve the final part of the paper \cite{GGT} showing
that the \emph{Lagrange multiplier} introduced to ``kill'' the
direction of dilation
invariance coming from the critical Sobolev exponent is indeed a
\emph{natural constraint} if we allow some invariance
(Kelvin invariance) on the
solutions. This lets us switch
from a local bifurcation result in \cite{GGT} to a global one.

This invariance is a good tool to overcome the degeneracy of critical
problems in $\R^N$ which are invariant under dilation
and can also be applied to the result in \cite{DGG}, where a Pohozaev
identity gives the result only locally.

Another technical problem arises since our
nonlinearities in general are not $\C^1$ at zero. This problem was already
noticed by~\cite{GLW} and indeed their existence results are given in
dimension $3$ where they are able to define and to invert the
linearized operator associated to their system.
To overcome this problem we use a different functional setting that allows us to work only with positive values of $u_1$ and $u_2$. Observe that the functional setting of our operator is a delicate part of the proof.

Secondly  we continue the study in \cite{GGT} and we
address to the existence of nonradial solutions to \eqref{0} using
in a tricky way some \emph{even} and \emph{odd} symmetries. Obviously
our solutions cannot be invariant with respect to odd symmetries since
we are looking for positive ones.
But we can introduce a suitable setting (see Eq.~\ref{mf11}) in which
we can make use of
this invariance.  This is a new aspect that has never been investigated
before and that can shed light on how solutions of systems of this
type are.

This use of the symmetries is the key point that allows us to
distinguish between
radial and nonradial solutions.

A crucial step of our method is the characterization of the kernel of the linearized operator associated to our systems.
Actually, in~\cite{GGT}, we find
radial solutions using the classical \textit{Crandall-Rabinowitz
Theorem} which requires a one dimensional
kernel.  This is achieved by restricting the problem to radially
symmetric functions and ``killing'' the direction of scale
invariance.

Considering also nonradial functions the 
dimension of the  kernel increases dramatically and it becomes very
hard to control it.  Moreover it is not clear whether the solution obtained
considering this new kernel is \emph{nonradial}.  As said before, the
use of suitable even and odd symmetries is significant and allows us to prove that in
many cases the kernel contains only nonradial functions and it is
\emph{odd} dimensional. To exploit them, we need some invariance on
the operator associated to our problem. This invariance naturally
appears
in the case of a~$2\times 2$ system while it not clear whether it applies in the general case of more equations as \eqref{0}.  For this reason we focus hereafter on the case ~$2\times 2$ and we believe that a further study is needed to understand the general case. To compute the dimension of the kernel in these symmetric spaces we need 
a classification of symmetric spherical harmonics in $\IS^N$ 
and indeed this is part of Section \ref{s3} and~\ref{s4}.

Finally we also give an asymptotic expansion of the solutions near
the bifurcation point so as to better understand them. In this way we can distinguish different 
nonradial solutions by their symmetries and expansions.

\smallskip

\section{Statement of the main results}

Let us introduce our abstract setting. We consider
\begin{equation}
  \label{eq:system}
  \begin{cases}
    -\Delta u_1 = F_1(\a, u_1, u_2)  & \text{in }\R^N,\\[1\jot]
    -\Delta u_2 = F_2(\a, u_1, u_2)  & \text{in }\R^N,\\[1\jot]
    u_1\geq 0,\ u_2\geq 0,\quad
    u_1,u_2\in D^{1,2}(\R^N),
  \end{cases}
\end{equation}
where the $F_i$ satisfy the following assumptions:  for all $\a\in \R$ and
for $i = 1,2$,
\begin{enumerate}
  \setlength{\itemsep}{3.0pt plus 2.5pt minus 1.0pt}%
  \renewcommand{\theenumi}{F\arabic{enumi}}%
  \renewcommand{\labelenumi}{(\theenumi)}%
\item\label{F-cont} the derivatives $\partial_\a F_i$,
  $\partial_{u} F_i$ and $\partial_{\a u} F_i$ of the map
  $F_i : \R \times (0, +\infty)^2 \to \R : (\a, u) \mapsto F_i(\a, u)$
  exist and are continuous;
\item\label{F-integ} for all $\a \in \R$, there exists a neighborhood
  $\mathcal A$ of $\a$ and a constant $C$ such that, for all
  $\a \in \mathcal A$ and $(u_1, u_2) \in (0, +\infty)^2$,
  $\abs{\partial_u F_i(\a, u_1, u_2)}
  \le C (u_1^{2^* - 2} + u_2^{2^*-2})$ and
  $\abs{\partial_{\a u} F_i(\a, u_1, u_2)}
  \le C (u_1^{2^* - 2} + u_2^{2^*-2})$;
\item\label{F-1} $F_i(\a, 1,1) = 1$;
\item\label{F-crit} $F_i(\a, \lambda u_1, \lambda u_2)
  = \lambda^{2^*-1} \, F_i(\a, u_1, u_2)$ for all $\lambda > 0$ and
  $(u_1, u_2) \in (0, +\infty)^2$;
\item\label{F-symm} $F_1(\a, u_1, u_2) = F_2(\a, u_2, u_1)$
  for all $(u_1, u_2) \in (0, +\infty)^2$;
\item\label{F-transversal}\label{F-last}
  for all $\a$, $\partial_\a \b(\a) > 0$ where
  $\b(\alpha) :=
  \partial_{u_1} F_1(\a, 1,1) - \partial_{u_2} F_1(\a, 1,1)$.
\end{enumerate}
By \eqref{F-1} it is straightforward that System~\eqref{eq:system}
admits, for
any $\a\in\R$, the trivial solution $(u_1, u_2) = (U,U)$ and
\eqref{F-crit} says that our system is scale invariant. Further,
in view of Eq.~\eqref{eq:system}, it is also translation invariant.

This generalization encompasses the following \emph{Schrodinger}
system
\begin{equation}
  \label{scr}
  \begin{cases}
 -\Delta u_1=\a u_1^{2^*-1} + (1-\a) u_1^pu_2^{2^*-1-p}
    & \text{in }\R^N,\\[1\jot]
    -\Delta u_2=(1- \a) u_1^{2^*-1-p}u_2^p+\a u_2^{2^*-1} 
    & \text{in }\R^N,\\[1\jot]
    u_1\geq 0,\ u_2\geq 0, \quad  
    u_1,u_2\in D^{1,2}(\R^N),
  \end{cases}
\end{equation} 
with $0\leq p<2^*-1$ and $\a$ is a real parameter. When $p=0$ System \eqref{scr} becomes
\begin{equation}
  \label{1}
  \begin{cases}
 -\Delta u_1=\a u_1^{2^*-1} + (1-\a) u_2^{2^*-1}
    & \text{in }\R^N,\\[1\jot]
    -\Delta u_2=(1- \a) u_1^{2^*-1}+\a u_2^{2^*-1} 
    & \text{in }\R^N,\\[1\jot]
    u_1\geq 0,\ u_2\geq 0, \quad  
    u_1,u_2\in D^{1,2}(\R^N),
  \end{cases}
\end{equation} 
 while for $p=\frac 2{N-2}$ we get System \eqref{0b}. Moreover System \eqref{eq:system} includes System \eqref{dh}.

Our first result is the generalization of the local radial bifurcation
result obtained in \cite{GGT} for  \eqref{1} to a global one for System~\eqref{eq:system}.
An important role in our results will be played  by  the Jacobi
polynomials  $P_j^{(\beta, \gamma)}$ that we introduce now. They are
defined as
  \begin{equation}
    P_m^{(\beta, \gamma)}(\xi)
    = \sum_{s=0}^m \binom{m+\beta}{s} \binom{m+\gamma}{m-s}
    \left(\frac{\xi-1}2\right)^{m-s} \left(\frac{\xi+1}2\right)^s 
  \end{equation}
for $m\in \N$, $\beta,\gamma\in \R^+$ and $\xi\in \R$.
\begin{theorem}\label{radial}
  Assume \eqref{F-cont}--\eqref{F-last}.
  The point $(\a^*, U,U)$ is a \emph{radial} bifurcation point from
  the curve of trivial solutions $(\a,U,U)$ to
  System~\eqref{eq:system} if $\a^*$ satisfies
  \begin{equation}
    \label{eq:degen}
    \b(\a^*)
    = \frac{(2n+N)(2n+N-2)}{N(N-2)}
  \end{equation}
  for some $n \in \N$, where $\b$ is defined
  in~\eqref{F-transversal}.
  More precisely there exists a continuously differentiable curve
  defined for $\e$ small enough
  \begin{equation}\nonumber
    (-\e_0,\e_0) \to \R \times \bigl(D^{1,2}_\rad(\R^N)\bigr)^2 :
    \e \mapsto \bigl(\a(\e), u_1(\e), u_2(\e)\bigr)
  \end{equation}
  passing through $(\a^*,U,U)$, i.e.,
  $\bigl(\a(0),u_1(0), u_2(0)\bigr) = (\a^*,U,U)$,
  such that, for all $\e \in (-\e_0,\e_0)$,
  $(u_1(\e), u_2(\e))$ is a radial solution to~\eqref{eq:system} 
  with $\a = \a(\e)$. Moreover,
  \begin{equation}\label{c-curve}
    \begin{cases}
      u_1(\e)
      = U + \e W_n(|x|) + \e \phi_{1,\e}(|x|),\\
      u_2(\e)
      = U - \e W_n(|x|) + \e \phi_{2,\e}(|x|),
    \end{cases}
  \end{equation}
  with $W_n$ being the function  
  \begin{equation}\label{W-n}
    \begin{split}
      W_n(|x|)
      & := \frac 1{\left(1+|x|^2\right)^{\frac{N-2}2}} \,
      P_n^{\left(\frac{N-2}{2}, \frac{N-2}{2} \right)}
      \left(\frac {1-|x|^2}{1+|x|^2}\right) 
    \end{split}
  \end{equation}
 where $\phi_{1,\e},\phi_{2,\e}$ are functions uniformly bounded in
  $D^{1,2}(\R^N)$ with respect to $\e \in (-\e_0,
  \e_0)$, and such that $\phi_{i,0}=0$ for $i=1,2$. Finally the
  bifurcation is global and the Rabinowitz alternative holds.
\end{theorem}

The values $\a^*$ in \eqref{eq:degen} are all of those for which the
linearized system at the trivial solution $(U,U)$ is non-invertible
showing that condition \eqref{eq:degen} is also necessary.

\begin{corollary}\label{result-scr}
  For any $n\in \N$, let 
  \begin{equation}\label{3}
    \a^*_{n}=\begin{cases} 
      \frac{(2n+N-2)(2n+N)}{2N(N+2-p(N-2))}
      + \frac{N+2}{2 (N+2-p(N-2))} - \frac {p(N-2)}{N+2-p(N-2)}
      &\text{in }~\eqref{scr},\\[3\jot]
      \frac{(2n+N-2)(2n+N)}{2N^2} + \frac{N-2}{2N}
      &\text{in }~\eqref{0b},\\[3\jot]
      \frac{(2n+N-2)(2n+N)+N(6-N)}{8N}
      &\text{in }~\eqref{dh}.
    \end{cases} 
  \end{equation}
  Then $(\a^*, U,U)$ is a radial bifurcation point of
  Systems~\eqref{scr}, \eqref{0b} and \eqref{dh} from its curve of trivial solutions $(\a,U,U)$ if
  $\a^* = \a^*_n$ for some $n\in\N$.
  Moreover, the expansion around the bifurcation point given by
  Theorem~\ref{radial} holds and the curve is global.
\end{corollary}
\begin{remark}
An interesting fact is that in~\eqref{scr} the
exponent $p$ does not
enter in a relevant way in the proof of the previous results and indeed the solutions we find have,
near a bifurcation point,  the same expansion for every value of
$p$. In this way we have a path of
solutions connecting \eqref{1} with  \eqref{0b} showing that these solutions are not due the variational structure of \eqref{1}.
\end{remark}

The next step is to find nonradial solutions. In \cite{GL} was proved
that in the cooperative case (i.e., when $1-\a>0$),
System~\eqref{scr} admits only radial solutions.
Note that, for all $n \ge 1$, $1 - \a^*_n < 1 - \a^*_1 = 0$
where $\a^*_n$ is defined by~\eqref{3}. Then  $\a^*_n$ are good
``candidates'' to find nonradial solutions. Moreover,
at each value $\a^*_n$ the linearized system possesses many nonradial
solutions and the kernel
becomes richer and richer as $n\to \infty$ (see
Proposition~\ref{prop-lin}).
 However, one technical problem in looking for nonradial solutions is that the
kernel of the linearized problem at a degeneracy point always
contains the radial function $W_n$ defined by~\eqref{W-n}.  So our aim
becomes to choose a suitable
subspace of the kernel in which $W_n$ does not lies. This will
be done by using in a tricky way some odd-symmetries.  It is possible
indeed to apply such symmetries to a linear combination of the
components $u_1,u_2$ even if the solutions we are interested in are
positive.

Here is our basic idea: if one writes
\begin{equation}\label{mf11}
  \begin{cases}
    u_1 = U+\frac{z_1 +z_2}2 \\
    u_2 =  U+\frac{z_1 -z_2}2
  \end{cases}
\end{equation}
then the system satisfied by $z_1,z_2$ admits solutions obtained by imposing the following symmetries
on $(z_1, z_2)$:
\begin{equation}\label{mf12}
  \forall (x', x_N) \in \R^N,\qquad
  z_1(x', x_N) = z_1(\abs{x'}, -x_N)
  \text{ and }
  z_2(x', x_N) = -z_2(\abs{x'}, -x_N) ,
\end{equation}
(more general symmetries will be imposed later; see
Section~\ref{sec:first-case} for more details).
The \emph{crucial} remark is that the new system in $(z_1,z_2)$
obtained by \eqref{mf11} is  \emph{invariant} for the symmetries in
\eqref{mf12} (see \eqref{2.0}--\eqref{2.3-b}).
This use of odd symmetries  is unclear if we considered directly
System \eqref{eq:system}.

In order to state our first nonradial bifurcation
result, we use in $\R^N$ the spherical coordinates
$(r,\varphi,\theta_1,\dots,\theta_{N-2})\in [0,+\infty)\times
[0,2\pi)\times[0,\pi)^{N-2}$.
We have
\begin{theorem}\label{teo1}
  Assume \eqref{F-cont}--\eqref{F-last} and let $\a^*_n$ be the
  unique solution to~\eqref{eq:degen} for some $n \in \N$.
  The point $(\a^*_n, U,U)$ is a \emph{nonradial} bifurcation point
  for the curve of
  trivial solutions $(\a,U,U)$ to System \eqref{eq:system} when
  $n \bmod 4 \in \{1,2\}$. More precisely, there exist a
  \emph{continuum} $\mathcal{C}$ of nonradial solutions $(u_1,u_2)$ 
  to System~\eqref{eq:system},
  bifurcating from $(\a^*_n,U,U)$; the bifurcation is global
  and the Rabinowitz alternative holds.  Finally for
  any sequence of solutions $(\a_k, u_{1,k}, u_{2,k}) \to
  (\a^*_n, U, U)$, we have that (up to a subsequence)
 \begin{equation}\label{num0}
    \begin{cases}
      u_{1,k}
      = U + \e_k Z_n\left(x\right)
      + o(\e_k) ,\\
      u_{2,k}
      = U - \e_k Z_n\left(x\right)
      + o(\e_k) ,
      \end{cases}
    \end{equation}
    as $k\to \infty$ where $\e_k = \|z_{2,k}\|_X \to 0$
    (see \eqref{mf11} and \eqref{eq:defX})
    and $Z_n\not\equiv0$ is the function
    \begin{equation}\label{num1}
      Z_n(x) = \hspace{-3pt}\sum_{h=1,\ h \text{\upshape\ odd}}^n a_h
      \frac{r^h}{(1+r^2)^{h+\frac{N-2}2}}
      P_{n-h}^{\left(h+\frac{N-2}2,\thinspace h+\frac{N-2}2\right)}
      \left(\frac{1-r^2}{1+r^2}\right)
      P_h^{\left(\frac{N-3}2,\frac{N-3}2\right)} (\cos\theta_{N-2})
    \end{equation}
    for some coefficients $a_h\in \R$. 
\end{theorem}
\noindent Observe that the functions
$P_h^{\left(\frac{N-3}2,\frac{N-3}2\right)} (\cos\theta_{N-2})$ are
the spherical harmonics that are $O(N-1)$-invariant.
\begin{corollary}\label{cor:scr}
  Let $n\in\N$ and $\a^*_n$ as defined in Corollary \ref{result-scr}.
  Then the same
  claims of Theorem~\ref{teo1} hold for Systems~\eqref{scr} and~\eqref{dh}.
\end{corollary}
It is possible
to prove a similar result using more symmetries. Here we ask the
following ones:
$\forall x = (x', x_{N-m+1}, \dotsc, x_N)\in \R^{N-m}\times \R^m$,
  \begin{align*}
      & z_1(x) = z_1(\abs{x'},\pm x_{N-m+1}, \dotsc,\pm x_N),
      \text{ and }\\
      &z_2(x', x_{N-m+1}, x_N) = - z_2(|x'|, -x_{N-m+1}, \dotsc,x_N),\\
     & \cdots\\
      &z_2(x', x_{N-m+1}, x_N) = - z_2(|x'|, x_{N-m+1}, \dotsc,- x_N)
       \bigr\}.
\end{align*}
Imposing these symmetries on the functions $z_1,z_2$ defined in \eqref{mf11}, we get the following result:
\begin{theorem}\label{teo3}
  Let $1\leq m\leq N$ and let $\a^*_n$ be the unique solution
  to~\eqref{eq:degen} for some $n\geq m$. Suppose that
  \begin{equation}\label{max}
    \binom{m+\left\lfloor\frac{n-m}2\right\rfloor}{m}
    \text{ is an odd integer.}
  \end{equation}  
  Then for any $m$ there exists a
  \emph{continuum} $\mathcal{C}_m$ of nonradial solutions
  that satisfies System~\eqref{eq:system},
  bifurcating from $(\a^*_n,U,U)$ and the bifurcation is global
  and the Rabinowitz alternative holds.
  Moreover the continua $\mathcal{C}_m$ are distinct and  we have
  that, up to a subsequence, $(u_1,u_2)$ has the same expansion as in
  \eqref{num0} where
   \begin{equation}
    Z_n(x)=\sum_{h=1}^n a_h  
    \frac{r^h}{(1+r^2)^{h+\frac{N-2}2}}
    P_{n-h}^{\left(h+\frac{N-2}2, h+\frac{N-2}2\right)}
    \left(\frac{1-r^2}{1+r^2}\right)
    Y_h(\theta)    
  \end{equation}
  and  the spherical harmonics $Y_h(\theta)$ are $O(N-m)$ invariant and
  odd in the last $m$ variables.
\end{theorem}

\begin{corollary}
  Let $n\in\N$ and $\a^*_n$ as defined in Corollary~\ref{result-scr}.
  Then the same
  claims of Theorem~\ref{teo3} hold for System~\eqref{scr} and~\eqref{dh}.
\end{corollary}
For the reader's convenience, we state the previous theorem when $m=2$.
\begin{corollary}\label{teo2}
  Let $m=2$ in Theorem \ref{teo3}. Then if
   \begin{equation}\label{max2}
 n \bmod 8 \in \{2,3,4,5\}
 \end{equation}    
  the claim of Theorem  \ref{teo3} holds and $Z_n$ in this case is given by
  \begin{equation}
    Z_n(x)=\sum_{h=1}^n a_h  
    \frac{r^h}{(1+r^2)^{h+\frac{N-2}2}}
    P_{n-h}^{\left(h+\frac{N-2}2, h+\frac{N-2}2\right)}
    \left(\frac{1-r^2}{1+r^2}\right)
    Y_h(\theta)    
  \end{equation}
  for some coefficients $a_h\in \R$, where $Y_h(\theta)$ are
  spherical harmonics which are
  $O(N-2)$ invariant and are odd with respect to $x_N$ and to
  $x_{N-1}$.
\end{corollary}
We conclude by giving one more existence result which produces a nonradial
solutions for every value of $n$. These solutions are found imposing
an odd symmetry with respect to an angle in spherical coordinates and
also a periodicity assumption. They are different from the
previous ones since they have a different expansion.

\begin{theorem}\label{teo4}
  Assume \eqref{F-cont}--\eqref{F-last} and $\a^*_n$
  be the unique solution to~\eqref{eq:degen} for some $n\in\N$.
  Then for any $n\in\N$, $n \ge 2$, there exists a
  \emph{continuum} ${\mathcal D}_n$ of nonradial solutions 
  to System~\eqref{eq:system},
  bifurcating from $(\a^*_n, U,U)$. When $\e$ is small enough this
  continuum is a continuously differentiable curve
  \begin{equation}\nonumber
    (-\e_0,\e_0) \to \R \times \bigl(D^{1,2}_\rad(\R^N)\bigr)^2 :
    \e \mapsto \bigl(\a(\e), u_1(\e), u_2(\e)\bigr)
  \end{equation}
  passing through $(\a^*_n,U,U)$, i.e.,
  $\bigl(\a(0),u_1(0), u_2(0)\bigr) = (\a^*_n,U,U)$,
  such that, for all $\e \in (-\e_0,\e_0)$,
  $(u_1(\e), u_2(\e))$ is a nonradial solution to~\eqref{eq:system} 
  with $\a = \a(\e)$.
  Moreover,
  \begin{equation*}
    \begin{cases}
      u_1(\e)
      = U + \e Z_n(x) + \e \phi_{1,\e}(x),\\
      u_2(\e)
      = U - \e Z_n(x) + \e \phi_{2,\e}(x),
    \end{cases}
  \end{equation*}
  with
  \begin{equation}\label{ancoraZ}
    Z_n(r,\varphi,\Theta)
    = a \frac{r^n}{\left(1+r^2\right)^{n+\frac{N-2}2}}
    \sin(n\varphi)(\sin\theta_1)^n\cdots (\sin\theta_{N-2})^n,
    \qquad a\in \R,
  \end{equation}
  (here we use the spherical coordinates
  $(r, \varphi, \Theta) = (r,\varphi,\theta_1,\dots,\theta_{N-2})$ in
  $\R^N$). Moreover the bifurcation is global and the Rabinowitz
  alternative holds.
\end{theorem}
\begin{remark}
  Note that the function
  $Y_n(\varphi,\Theta) = \sin(n\varphi)(\sin\theta_1)^n\cdots
  (\sin\theta_{N-2})^n$
  is the unique spherical harmonic of order $n$ which is odd and
  periodic of period $\frac{2\pi}n$ with respect to the angle
  $\varphi$. Moreover, in Cartesian coordinates we have that
  $Y_n(x) = \Im(x_1+ix_2)^n$.
\end{remark}
\begin{corollary}
  Let $n\in\N$ and $\a^*_n$ as defined in Corollary \ref{result-scr}.
  Then the same
  claims of Theorem~\ref{teo4} hold for System~\eqref{scr} and~\eqref{dh}.
\end{corollary}

\begin{remark}
  It is difficult to give a formula with the exact number of solutions
  which takes in account all the previous theorems. Here we describe a
  particular case: choose $n=4$ in \eqref{eq:degen} and $N\ge4$ then
  we have the existence of at least {\em five} solutions bifurcating
  by $(U,U)$ as follows:
  \begin{enumerate}
  \item[\textit{i})] one radial solution (Theorem \ref{radial}),
  \item[\textit{ii})] one nonradial solution with $z_1$ even in all the
    coordinates and $z_2$ odd with respect to $x_{N-1}$ and $x_N$ and
    even in other coordinates (Corollary \ref{teo2}),
  \item[\textit{iii})] one nonradial solution with $z_2 $ odd with
    respect to $x_{N-3},\dots,x_N$ and even in other coordinates
    (Theorem \ref{teo3} with $m=3$),
  \item[\textit{iv})] one nonradial solution in $\R^N$ with $N\geq 4$ with
    $z_2 $ odd with respect to $x_{N-4},\dots,\linebreak[2] x_N$
    and even in other
    coordinates (Theorem \ref{teo3} with $m=4$),  
  \item[\textit{v})] one nonradial solution where $z_1$ and $z_2$ are
    periodic of period $\frac{2\pi}4$ with respect to the angle
    $\varphi$ and $z_2$ is odd in $\varphi$ (Theorem \ref{teo4}).
  \end{enumerate}
  In the following table, which does not pretend to be exhaustive, we
  show the number of solutions bifurcating from $(U,U)$ arising from
  Theorems \ref{radial}--\ref{teo4}.\vskip0.2cm
  \begin{equation*}
    \begin{array}{r|ccc}
      &N=3&N=4&N=5\\ \hline
      n=2&4&4&4\\ \rowcolor[gray]{0.85}
      n=3&4&4&4\\
      n=4&4&5&5\\ \rowcolor[gray]{0.85}
      n=5&4&5&6\\
      n=6&3&4&5\\ \rowcolor[gray]{0.85}
      n=7&2&3&3
    \end{array}
  \end{equation*}
\end{remark}
\par
\begin{remark}
  Note that the our results for System \eqref{0b} hold for any
  dimension $N\geq 3$, extending some recent results of~\cite{GLW}.
  Finally, as observed in~\cite{GLW}, when the dimension $N\ge 4$,
  System~\eqref{0b} becomes linear or sublinear in some of its
  components and this fact produces problem in defining and estimating
  the linearization.  In some sense, we can say that the bifurcation
  theory suits well this problem.
  We remark moreover that the  solutions founded in~\cite{GLW}  are
  always different from ours since their expansion is of the following
  type $u_1=U+\e \phi_1$ and $u_2=\sum_k U_{\delta_k,y_k}+\e \phi_2$.
\end{remark}

The paper is organized as follows: in Section \ref{s2} we recall some
preliminaries and introduce the functional setting to find the
nonradial solution. In Section \ref{s3} we define the symmetric spaces
and prove Theorems \ref{radial}, \ref{teo1} and \ref{teo3}. In Section \ref{s4} we prove Theorem \ref{teo4}. 

\section{Preliminary results and the functional setting}
\label{s2}

To study System~\eqref{eq:system}, we perform the following change of variables
\begin{equation}\label{2.0}
  \begin{cases}
    z_1 = u_1 + u_2 - 2U, \\
    z_2 = u_1 - u_2,
  \end{cases}
\end{equation}
that turns \eqref{eq:system} into 
the system
\begin{equation}
  \label{2}
  \begin{cases}
    -\Delta z_1 = f_1(|x|,z_1,z_2) & \text{in }\R^N,\\
    -\Delta z_2 = f_2(|x|,z_1,z_2) & \text{in }\R^N,\\
    z_1,z_2\in D^{1,2}(\R^N),
  \end{cases}
\end{equation}
where
\begin{align}
  f_1(|x|,z_1,z_2)
  &:= F_1\Bigl(\a, U+\frac{z_1+z_2}{2},\ U+\frac{z_1-z_2}{2} \Bigr)
    \nonumber\\
  &\hspace{4em}
    + F_2\Bigl(\a, U+\frac{z_1+z_2}{2},\ U+\frac{z_1-z_2}{2} \Bigr)
    - 2 U^{2^*-1}  ,
    \label{f1}\\
  f_2(|x|,z_1,z_2)
  &:= F_1\Bigl(\a, U+\frac{z_1+z_2}{2},\ U+\frac{z_1-z_2}{2} \Bigr)
    \nonumber\\
  &\hspace{4em}
    - F_2\Bigl(\a, U+\frac{z_1+z_2}{2},\ U+\frac{z_1-z_2}{2} \Bigr).
    \label{f2}
\end{align}
One important feature in looking for nonradial solutions is that, using~\eqref{F-symm},  this
change of variables gives the following
invariance:
\begin{equation}\label{2.3-b}
  \begin{split}
    f_1(|x|,z_1,-z_2) &= f_1(|x|,z_1,z_2),\\
    f_2(|x|,z_1,-z_2) &= -f_2(|x|,z_1,z_2).
  \end{split}
\end{equation}
Solutions to~\eqref{eq:system} are zeros of the operator
\begin{equation*}
  T(\a,z_1,z_2)
  :=
  \begin{pmatrix}
    z_1 - (-\Delta)^{-1} \bigl(f_1(|x|,z_1,z_2)\bigr)\\[2\jot]
    z_2 - (-\Delta)^{-1} \bigl(f_2(|x|,z_1,z_2)\bigr)  
  \end{pmatrix}.
\end{equation*}
Clearly, $T(\a, 0,0) = (0,0)$ for all $\a \in \R$
(thanks to \eqref{F-1} and \eqref{F-crit}).
A necessary condition for the bifurcation is that the linearized operator $\partial _zT(\alpha, 0,0)$ is not invertible. 
This corresponds to study the system: 
\begin{equation}
  \label{prima-linearization}
  \begin{cases}
    -\Delta w_1 = \frac{\partial f_1}{\partial z_1}(|x|,0,0) \, w_1
    + \frac{\partial f_1}{\partial z_2}(|x|,0,0) \, w_2
    & \text{in }\R^N,\\[1\jot]
    -\Delta w_2=\frac{\partial f_2}{\partial z_1}(|x|,0,0) \, w_1
    + \frac{\partial f_2}{\partial z_2}(|x|,0,0) \, w_2
    & \text{in }\R^N,\\[1\jot]
    w_1,w_2 \in D^{1,2}(\R^N).
  \end{cases}
\end{equation}
A simple computation shows 
\begin{align*}
  \frac{\partial f_1}{\partial z_1}(\alpha,0,0)
  &=\frac 12\left[\frac{\partial F_1}{\partial u_1}(\alpha,U,U)
    + \frac{\partial F_1}{\partial u_2}(\alpha,U,U)
    + \frac{\partial F_2}{\partial u_1}(\alpha,U,U)
    + \frac{\partial F_2}{\partial u_2}(\alpha,U,U)\right],\\[2\jot]
  \frac{\partial f_1}{\partial z_2}(\alpha,0,0)
  &=\frac 12\left[\frac{\partial F_1}{\partial u_1}(\alpha,U,U)
    - \frac{\partial F_1}{\partial u_2}(\alpha,U,U)
    + \frac{\partial F_2}{\partial u_1}(\alpha,U,U)
    - \frac{\partial F_2}{\partial u_2}(\alpha,U,U)\right],
\end{align*}
and a very similar expression holds for $\frac{\partial f_2}{\partial z_i}(\alpha,0,0)$ for $i=1,2$. First observe that from \eqref{F-symm} we get
\[\frac{\partial F_1}{\partial u_1}(\alpha,U,U)=\frac{\partial F_2}{\partial u_2}(\alpha,U,U)\quad \text{ and } \quad \frac{\partial F_1}{\partial u_2}(\alpha,U,U)=\frac{\partial F_2}{\partial u_1}(\alpha,U,U).\]
Then, differentiating \eqref{F-crit} with
respect to $\lambda$ we get
\[(\partial_{u_1}F_1 + \partial_{u_2}F_1)(\a, U,U) = (2^*-1) U^{2^*-2}
F_1(\a, 1, 1) = \frac{N+2}{N-2} U^{\frac{4}{N-2}}.\]
Moreover, using again~\eqref{F-crit}:
\begin{equation*}
  \partial_{u_j}F_i(\a, \lambda u_1, \lambda u_2)
  = \lambda^{2^*-2} \partial_{u_j}F_i(\a,  u_1,  u_2) 
  \quad
  \text{for } i = 1,2 \text{ and } j = 1,2,
\end{equation*}
and in particular
\[\partial_{u_j}F_i(\a, U, U) = U^{2^*-2} \, \partial_{u_j}F_i(\a, 1,1).\]
Putting together all these remarks, it is straightforward that system
\eqref{prima-linearization} becomes
\begin{equation}
  \label{linearization}
  \begin{cases}
    -\Delta w_1 = \frac{N+2}{N-2} \, U^{\frac{4}{N-2}} w_1
    & \text{in }\R^N,\\[1\jot]
    -\Delta w_2=\b(\a) \, U^{\frac{4}{N-2}} w_2
    & \text{in }\R^N,\\[1\jot]
    w_1,w_2 \in D^{1,2}(\R^N),
  \end{cases}
\end{equation}
with $\b(\a)$ defined in~\eqref{F-transversal}.

System \eqref{linearization} is degenerate for any $\a$, since the
problem is invariant by translations and dilations. Indeed, it is well
known that the first equation admits the solutions
$W(x):=\frac{1-|x|^2}{\left(1+|x|^2\right)^{N/2}}$ and
$W_i(x)=\frac {\partial U}{\partial x_i}$ for $i=1,\dots, N$.
The second equation instead has solutions if and only if $\b(\a) $ is
an eigenvalue of the linearized equation of the classical critical
problem at the standard bubble $U$.  Using the classification of the
eigenvalues and eigenfunctions in \cite[Theorem 1.1]{GGT}, one gets that
the second equation admits nontrivial solutions if and only if
$\b(\a) = \l_n \frac{N+2}{N-2}$
with $\l_n := \frac{(2n+N-2)(2n+N)}{N(N+2)}$\label{lambda_n}
for some $n\in \N$.  So
we have the following classification result for~\eqref{linearization}.
\begin{proposition}
  \label{prop-lin}
  Let $\beta_n$ be given by
  \begin{equation}
    \label{eq:bn}
    \b_n := \frac{(2n+N)(2n+N-2)}{N(N-2)} .
  \end{equation}
  \begin{itemize}
  \item[i)] When $\b(\a) \ne \b_n$ for all $n \in \N$, all solutions 
    to \eqref{linearization} are given~by
    \begin{equation}\label{primo-caso}
      (w_1,w_2)
      = \left(\sum_{i=1}^N a_i\frac{\partial U}{\partial x_i} + b W, \
        0\right)
    \end{equation}
    for some real constants $a_1,\dots,a_N,b$, where $W$ is the radial
    function defined by
    \begin{equation}
      \label{eq:W}
      W(x) :=\frac 1 d \left( x\cdot \nabla U + \frac{N-2}{2} U\right)
      =  \frac{1 - \abs{x}^2}{(1 + \abs{x}^2)^{N/2}}
    \end{equation}
    with $d := \tfrac{1}{2} N^{(N-2)/4} (N-2)^{(N+2)/4}$.
  \item[ii)] When $\b(\a) = \beta_n$ for some $n \in \N$, all solutions 
    to \eqref{linearization} are given by
    \begin{equation}
      (w_1,w_2)
      = \left(\sum_{i=1}^N a_i\frac{\partial U}{\partial x_i} + b W, \
        \sum_{k=0}^n A_k W_{n,k}(r) Y_k(\theta) \right)
    \end{equation}
    for some real constants $a_1,\dots,a_N,b,
    A_{0},\dots,A_n$, where
    $W_{n,k}$ are  
    \begin{equation}
    \label{sol-lin}
    W_{n,k}(r)
    := \frac{r^k}{(1+r^2)^{k+\frac{N-2}2}}
    P_{n-k}^{\left(k+\frac{N-2}2,\thinspace k+\frac{N-2}2\right)}
      \left(\frac{1-r^2}{1+r^2}\right)
  \end{equation}
  for $k=0,\dots,n$. Here, as usual, $Y_k(\theta)$ denotes a
  spherical harmonic related to the eigenvalue $k(k+N-2)$ and
  $P_j^{(a, b)}$ are the Jacobi polynomials.
  \end{itemize}
\end{proposition}

In \cite{GGT} we restricted to the radial
functions and  since the kernel of the second equation in \eqref{linearization}
at the values $\b_n$ is  one dimensional,
Crandall-Rabinowitz' Theorem allowed us to prove the bifurcation
result.  In the nonradial setting, the
kernel of the second equation in~\eqref{linearization}
is very rich.  We prove a bifurcation result using the Leray
Schauder degree, when this kernel has an odd dimension.

Of course, in this case, we need some compactness
of the operator $T$.  Since we seek positive solutions to
System~\eqref{eq:system} and the maximum principle does not apply,
the standard space $D^{1,2}(\R^N)$ does not seem to be the best one.  For this reason we use a suitable weighted
functional space.  Set
\begin{equation*}
  D := \Bigr\{u\in L^{\infty}(\R^N) \Bigm|
  \sup\limits_{x\in\R^N}\frac{|u(x)|}{U(x)}<+\infty \Bigr\}
\end{equation*}
endowed with the norm $\norm{u}_D :=
\sup_{x\in\R^N}\frac{|u(x)|}{U(x)}$ and define
\begin{equation}
  \label{eq:defX}
  X := D^{1,2}(\R^N)\cap D.
\end{equation}
Then $X$ is a Banach space when
equipped with the norm $\norm{u}_X := \max\{\norm{u}_{1,2},
\norm{u}_D\}$ where $\norm{u}_{1,2} = (\int_{\R^N} \abs{\nabla u}^2)^{1/2}$
  is the classical norm in $D^{1,2}(\R^N)$.

\begin{definition}
  Let us denote by $\X$ the space
  \begin{equation*}
    \X := \bigl\{ (z_1, z_2) \in X^2 \bigm|
    \exists \delta > 0,\   \abs{z_2} \le (2-\delta)U + z_1 \bigr\}
  \end{equation*}
  and define the operator
  \begin{equation*}
    T : \R \times \X \to X\times X
  \end{equation*}
  as
  \begin{equation}
    \label{T}
    T(\a,z_1,z_2) :=
    \begin{pmatrix}
      z_1 - (-\Delta)^{-1} \bigl(f_1(|x|,z_1,z_2) \bigr) \\[2\jot]
      z_2 - (-\Delta)^{-1} \bigl(f_2(|x|,z_1,z_2) \bigr)
    \end{pmatrix} .
  \end{equation}
\end{definition}

\medskip

Note that if $(z_1,z_2) \in \X$, both quantities $U+\frac{z_1+z_2}2$
and $U+\frac{z_1-z_2}2$ are positive so that
$F_i(\a, U+\frac{z_1+ z_2}{2}, U+\frac{z_1- z_2}{2})$ are well defined
on $\R^N$ and $C^1$.  Moreover, $\X$ is an open subset of $X^2$.

The zeros of the operator $T$ correspond to the solutions to
System~\eqref{eq:system}.  
As said before, Problem~\eqref{eq:system} is degenerate for any $\a$.
To overcome this degeneracy we will use some symmetry and invariance
properties. The solutions we will find will inherit the
symmetry and the invariance.
To overcome the degeneracy of the first
equation in~\eqref{linearization}, which is due to the scale
invariance of the problem, we use the \emph{Kelvin transform} $k(z)$
of $z$, namely
\begin{equation}\label{Kelvin}
  k(z)(x) := \frac{1}{|x|^{N-2}} \, z\left(\frac x{|x|^2}\right)
\end{equation}
and we denote by $X_k^{\pm }\subseteq X$ the subset of functions in
$X$ which are invariant (up to the sign) by a Kelvin transform, i.e.
\begin{equation}\label{def-X_k}
  X_k^{+} := \{z\in X \mathrel| k(z) = z \}
  \quad\text{and}\quad
  X_k^{-} := \{z\in X \mathrel| k(z) = -z \} .
\end{equation}
Observe that $U\in X_k^+$, $W \in X_k^-$ and, using the fact that
the Jacobi polynomials $P_j^{(a, b)}$ are \emph{even} if $j$
is even and \emph{odd} if $j$ is odd, an easy computation shows that
$W_{n,k} \in X_k^+$ if $n-k$ is \emph{even} while
$W_{n,k} \in X_k^-$ if $n-k$ is \emph{odd}.

\medskip
 
First we prove some properties of the operator $T$.

\begin{lemma}\label{lemma1}
  The operator $T$ given by~\eqref{T} is well defined and continuous
  from $\R\times \X$ to $X^2$.  Moreover, $\partial_\a T$, $\partial_z
  T$ and $\partial_{\a z} T$ exist and are continuous.
  Finally, $T$ maps
  $\R \times \bigl(\X \cap (X_k^+ \times X_k^{\pm}) \bigr)$ to
  $X_k^+\times X_k^{\pm}$.
\end{lemma}
\begin{proof}
  First notice that, \eqref{F-crit} implies
  $\lim_{\lambda\to 0} F_i(\a, \lambda u_1, \lambda u_2) = 0$.
  Thus, using~\eqref{F-integ}, one gets
  \allowdisplaybreaks
  \begin{align}
    \abs{F_i(\a, u_1, u_2)}
    &= \biggabs{\int_0^1
      \partial_\lambda\bigl( F_i(\a, \lambda u_1, \lambda u_2) \bigr)
      \intd \lambda}  \nonumber\\
    &\le \int_0^1 \bigabs{\partial_{u_1} F_i(\a, \lambda u_1, \lambda u_2)[u_1]
      + \partial_{u_2} F_i(\a, \lambda u_1, \lambda u_2)[u_2]}
      \intd\lambda \nonumber\\
    &\le C(u_1^{2^*-2} + u_2^{2^*-2}) (u_1 + u_2) \nonumber\\
    &\le C(u_1^{2^*-1} + u_2^{2^*-1}).
    \label{eq:F-integ2}
  \end{align}
  (Different occurrences of $C$ may denote different constants.)
  Given that $z_1,z_2$ and $U$ belong to $X$, \eqref{eq:F-integ2} implies
  that 
  $\bigabs{F_i\bigl(\alpha, U + \frac{z_1+z_2}{2},
    U + \frac{z_1-z_2}{2}\bigr)}\leq CU^{2^*-1}$
  and thus, using \eqref{f1} and \eqref{f2},
  \begin{equation*}
    \abs{f_i(\abs{x}, z_1, z_2)} \le CU^{2^*-1}\quad \text{ for }i=1,2.
  \end{equation*}
  Then $f_i(\abs{x}, z_1, z_2)$ belong to $L^{\frac{2N}{N+2}}(\R^N)$ and  there exists a unique
  $g_i\in D^{1,2}(\R^N)$ for $i=1,2$ such that $g_i$ is a weak 
  solution to 
    \begin{equation}\label{max3}
    -\Delta g_i = f_i(|x|,z_1,z_2)\quad\hbox{in }\R^N.
      \end{equation}
  The solution $g_i$ enjoys the following representation:
  \begin{equation}\nonumber\label{g_i}
    g_i(x) = \frac 1{\omega_N(N-2)}\int_{\R^N}
    \frac 1{|x-y|^{N-2}} f_i(\abs{y}, z_1,z_2)   \intd y
  \end{equation}
  where $\omega_N$ is the area of the unit sphere in $\R^N$.  
  This implies
  \begin{equation*}
    \left|g_{i}(x)\right|
    \le C\int_{\R^N}\frac 1{|x-y|^{N-2}}U^{2^*-1}(y) \intd y
    = CU(x)
  \end{equation*}
  and $g_i\in X$ showing that $T$ is well defined from $\X$ to
  $X\times X$.

  Next we have to show that the operator $T$ maps Kelvin invariant (up
  to a sign) functions into functions that are Kelvin invariant (with
  the same sign). It is enough to show that
  $\bigl( (-\Delta)^{-1} (f_1(|x|,z_1,z_2)), (-\Delta)^{-1}
  (f_2(|x|,z_1,z_2)) \bigr)$
  maps $\X \cap (X_k^+ \times X_k^{\pm})$ into
  $X_k^+\times X_k^{\pm} $.  Assume
  $(z_1,z_2)\in \X \cap (X_k^+ \times X_k^{\pm})$ and let, as before,
  $g_i = (-\Delta)^{-1}(f_i(|x|,z_1,z_2))$. Then $g_i\in X$
  is a weak solution to \eqref{max3} and
  letting $\widetilde g_i := k(g_i)$, the Kelvin transform of $g_i$ we
  have that $\widetilde g_i$ weakly solves
  \begin{equation*}
    -\Delta \widetilde g_i
    = - \frac{1}{\abs{x}^{N+2}}  \Delta g_i\Bigl(\frac{x}{\abs{x}^2}\Bigr)
    = \frac{1}{|x|^{N+2}} \,
    f_i\left(\frac{x}{|x|^2}, z_1\Bigl(\frac x{|x|^2}\Bigr),
      z_2\Bigl(\frac{x}{|x|^2}\Bigr) \right)
  \end{equation*}
  An easy consequence of \eqref{F-crit} is that
  \begin{align*}
    & \hspace{-2cm}\frac{1}{|x|^{N+2}} \,
      F_i\left(\alpha, \Bigl(U + \frac{z_1+z_2}{2}\Bigr)
      \Bigl(\frac{x}{|x|^2}\Bigr) ,
      \Bigl(U + \frac{z_1-z_2}{2}\Bigr)\Bigl(\frac{x}{|x|^2}\Bigr)
      \right)\\[1\jot]
    &=F_i\Bigl(\alpha,k(U)+\frac{k(z_1)+k(z_2)}{2},
      k(U)+\frac{k(z_1)-k(z_2)}{2}\Bigr).
  \end{align*}
  This, together with the fact that $U$ and $z_1$ are Kelvin invariant
  while $z_2$ is Kelvin invariant up to a sign (depending which space
  $X_k^{\pm}$ we are dealing with) shows that
  \begin{equation*}
    \frac{1}{|x|^{N+2}} \,
    f_i\left(\frac{x}{|x|^2}, z_1\Bigl(\frac x{|x|^2}\Bigr),
      z_2\Bigl(\frac{x}{|x|^2}\Bigr) \right)
    = f_i\bigl(|x|,z_1(x),\pm z_2(x)\bigr)
  \end{equation*}
  where $\pm$ depends on the space $X_k^{\pm}$ we consider.  Then, using
  \eqref{2.3-b}, it follows that
  \begin{equation*}
    \frac{1}{|x|^{N+2}} \,
    f_1\left(\frac 1{|x|}, z_1\Bigl(\frac x{|x|^2}\Bigr),
      z_2\Bigl(\frac x{|x|^2}\Bigr) \right)
    = f_1\bigl(|x|,z_1(x),z_2(x)\bigr)
  \end{equation*}
  while
  \begin{equation*}
    \frac 1{|x|^{N+2}} \,
    f_2\left(\frac{x}{|x|^2}, z_1\Bigl(\frac x{|x|^2}\Bigr),
      z_2\Bigl(\frac x{|x|^2}\Bigr)\right)
    = \pm f_2\bigl(|x|,z_1(x),z_2(x)\bigr).
  \end{equation*}
  This implies that $\widetilde g_1$ weakly solves
  $-\Delta \widetilde g_1=f_1(|x|,z_1(x),z_2(x))$ and $\widetilde g_2$
  solves $-\Delta \widetilde g_2=\pm f_2(|x|,z_1(x),z_2(x))$.  The
  uniqueness of solutions in $D^{1,2}(\R^N)$ then implies
  $ \widetilde g_1 = g_1$ and $ \widetilde g_2 = \pm g_2$ which shows
  that $g_1\in X_k^+$ and $g_2\in X_k^{\pm}$.  This concludes the
  first part of the proof.

  \medskip

  Let us now prove the continuity of $T$ on $\R \times \X$.  Let
  $\a_n \to \a$ in $\R$ and $(z_{1,n}, z_{2,n}) \to (z_1, z_2)$ in
  $\X$ as $n \to \infty$, and set
  \begin{equation*}
    g_{i,n} := (-\Delta)^{-1} f_{i,n} 
    \quad\text{where }
    f_{i,n}(x)
    := f_i(|x|,z_{1,n},z_{2,n})
    \text{ with } \a = \a_n.
  \end{equation*} 
  Since $z_{i,n} \to z_i$ in $D^{1,2}(\R^N)$, the convergence also
  holds in $L^{2^*}(\R^N)$.  Using \eqref{eq:F-integ2} and Lebesgue's dominated
  convergence theorem and its converse, one deduces that
  $f_{i,n} \to f_i$ in $L^{\frac{2N}{N+2}}$.  Therefore
  $g_{i,n} \to g_i$ in $D^{1,2}$ and $T(\a_n, z_n) \to T(\a, z)$
  in~$D^{1,2}$. Now let us show the convergence in $D$.  We have that
  \begin{equation}
    \begin{split}
      \frac{|g_{i,n}(x)-g_i(x)|}{U(x)}
      &\le
      \frac 1{\omega_N(N-2)U(x)}\int_{\R^N}
      \frac 1{|x-y|^{N-2 }}\frac{|f_{i,n}(y)-f_i(y)|}{U(y)^{2^*-1}}
      \, U(y)^{2^*-1} \intd y\\
      &\le C\sup\limits_{y\in\R^N} \frac{|f_{i,n}(y)-f_i(y)|}{U(y)^{2^*-1}}.
    \end{split}
  \end{equation}
  Moreover, using \eqref{F-crit}, one gets
  \begin{align*}
    \frac{\abs{f_{i,n}(y) - f_i(y)}}{U(y)^{2^*-1}}
    \le \sum_{i=1}^2
    \Bigl| & F_i\Bigl(\alpha_n, 1 + \frac{z_{1,n} + z_{2,n}}{2 U},\
             1 + \frac{z_{1,n} - z_{2,n}}{2 U} \Bigr) \\
           & \hspace{3em}
             - F_i\Bigl(\alpha, 1 + \frac{z_1 + z_2}{2 U},\
             1 + \frac{z_1 - z_2}{2 U} \Bigr) \Bigr|
  \end{align*}
  Thanks to the convergence in $D$, ${z_{j,n}}/{U} \to
  {z_{j}}/{U}$ uniformly for $j=1,2$.  Thus $1 + \frac{z_{1,n} \pm
    z_{2,n}}{2 U} \to 1 + \frac{z_{1} \pm z_{2}}{2 U}$ uniformly on~$\R^N$.
  The continuity of the maps $F_i$ then imply that both terms of the
  sum converge uniformly to~$0$.

  The existence and continuity of the derivatives is proved in a
  similar way.
\end{proof}

Next we show a compactness result for the operator
$(z_1,z_2) \mapsto \left((-\Delta)^{-1}f_1,(-\Delta)^{-1}f_2 \right)$.  Here
we need some
decay estimates on solutions of a semilinear elliptic equation.
\begin{lemma}[\cite{ST}]\label{lemma-ST}
  If  $0 < p < N$ and $h$ is a non negative, radial function
  belonging to $L^1(\R^N)$, then
  \begin{equation*}
    \int_{\R^N} \frac{h(y)}{|x-y|^{p}} \intd y
    = O\biggl( \frac{1}{\abs{x}^{p}} \biggr)
    \quad \text{as } |x| \to +\infty.
  \end{equation*}
\end{lemma}

Now we can prove our compactness result:

\begin{lemma}\label{lemma-comp}
  For all $\a$, the operator
  \begin{equation}
    M(z_1,z_2)
    := \left((-\Delta)^{-1} f_1(|x|,z_1,z_2) ,
      (-\Delta)^{-1}f_2(|x|,z_1,z_2) \right)
  \end{equation}
  is compact from $\X$ to $X^2$.
\end{lemma}
\begin{proof}
  \textit{1.\@}
  From Lemma~\ref{lemma1}, we have that
  $M : \X \to X^2$ is continuous.
  Now let $(z_n) = (z_{1,n},z_{2,n})$ be a bounded sequence in
  $\X$ and let us prove that, up to a subsequence,
  $g_n := M(z_n)$ converges strongly to some $g\in X\times X$.  On one
  hand, since $(z_n)$ is bounded in $D^{1,2}\times D^{1,2}$, going if
  necessary to a subsequence, one can assume that $(z_n)$ converges
  weakly to some $z=(z_1,z_2)$ in $D^{1,2}\times D^{1,2}$ and
  $z_n \to z$ almost everywhere.  On the other hand,
  $( \norm{z_n}_{D\times D} )$ is also bounded which means that
  $\abs{z_{i,n}} \le C U$ where $C$ is independent of $i$ and $n$ and
  so, using \eqref{eq:F-integ2},
  $\abs{f_i(\abs{x}, z_n)} \le C U^{2^*-1}$.  Lebesgue's dominated
  convergence theorem then implies that $ f_i(\abs{x}, z_n)$ converges strongly
  to $f_i(\abs{x}, z)$ in $L^{\frac{2N}{N+2}}$ for $i=1,2$.  From the
  continuity of $(-\Delta)^{-1} : L^{\frac{2N}{N+2}} \to D^{1,2}$, one
  concludes that $g_n \to g$ in $D^{1,2}\times D^{1,2}$.  The
  inequality $\abs{z_{i,n}} \le C U$ also implies
  \begin{equation*}
    \abs{g_{i,n}(x)}
    \le C \int_{\R^N} \frac1{|x-y|^{N-2}} \, \abs{f_i(z_n(y))} \intd y
    \le C \int_{\R^N}\frac{U^{2^*-1}(y)}{|x-y|^{N-2}} \intd y
    = C  \, U(x),
  \end{equation*}
  and passing to the limit yields $g_i \in D$.

  \textit{2.\@} It remains to show that $\norm{g_n - g}_{D\times D} \to 0$.
  First, H\"older's inequality allows to get the estimate:
  \begin{align*}
    \hspace{2em}
    &\hspace{-2em}
      \abs{g_{i,n}(x)-g_i(x)}\\
    &\le C\int_{\R^N} \frac{1}{|x-y|^{N-2}} \,
      \bigabs{f_i(\abs{y}, z_n(y))-f_i(\abs{y}, z(y))} \intd y
    \\
    &= C\int_{\R^N} \frac{U^{2^*-1-\e}(y)}{|x-y|^{N-2}} \,
      \frac{\bigabs{f_i(\abs{y}, z_n(y))-f_i(\abs{y}, z(y))}}{
      U^{2^*-1-\e}(y)} \intd y\\
    \displaybreak[2]
    &\le C \left(\mkern 5mu \int_{\R^N}
      \left|\frac{U^{\frac{N+2}{N-2}-\e}(y)}{|x-y|^{N-2}}\right|^{\frac{q}{q-1}}
      \right)^{\frac{q-1}q}
      \left(\mkern 5mu \int_{\R^N}
      \left|\frac{\bigabs{f_i(\abs{y}, z_n(y))-f_i(\abs{y}, z(y))}}{
      U^{2^*-1}(y)} \, U^{\e}(y)  \right|^q\right)^{\frac 1q}
  \end{align*}
  where $\e > 0$ will be chosen small and $q > 1$ large such that
  $\e q= 2^*$.
  Note that \eqref{eq:F-integ2} implies
  $\abs{f_i(\abs{y}, z(y))} \le C (U + \abs{z_1} + \abs{z_2})^{2^* -
    1} + C U^{2^*-1}$ and so the ratio in the right integral is bounded
  on~$\R^N$.  Thus the integrand of the right integral is bounded by
  $C^q\, U^{\e q}(y) \le C U^{2^*}(y) \in L^1(\R^N)$ where $C$
  is independent of~$n$.  Lebesgue's dominated convergence theorem
  then implies that this integral converges to $0$ as $n \to \infty$.

  The proof will be complete if we show:
  \begin{equation}\label{cl1}
    \begin{split}
      \int_{\R^N}
      \left|\frac {U^{\frac{N+2}{N-2}-\e}(y)}{|x-y|^{N-2 }}
      \right|^{\frac{q}{q-1}} \intd y
      \le \frac{C}{(1+|x|)^{(N-2)\frac q{q-1}}}
      = C U^{\frac{q}{q-1}}(x).
    \end{split}
  \end{equation}
  This inequality follows from Lemma \ref{lemma-ST} because
  $h := U^{\left(\frac{N+2}{N-2}-\e\right) \frac{q}{q-1}} \in L^1(\R^N)$
  i.e., $(N-2) \bigl( \frac{N+2}{N-2}-\e \bigr) \frac{q}{q-1} > N$,
  and $(N-2) \frac{q}{q-1} < N$
  are possible if $\e$ is small enough and
  $q$ is large enough.
\end{proof}
\section{The role of symmetries}
\label{s3}

The operator $T$ is a compact perturbation of
the identity and, as proved in Lemma~\ref{lemma1}, maps  $\R \times
\bigl(\X \cap (X_k^+ \times X_k^{\pm}) \bigr)$ into
  $X_k^+\times X_k^{\pm}$.

We want to find solutions to our problem as zeroes of $T$ and we will
use the bifurcation theory. As explained in the introduction, we want
to find both radial and nonradial solutions. In particular, to obtain
the nonradial ones, we use some symmetry properties of the operator $T$
that can be obtained by \eqref{2.3-b}.

We
state the definition in a general way and we will then apply to some
specific cases so to obtain different solutions. Let us introduce some
notations.  Let $\mathcal{S}$ be a subgroup of $O(N)$, where $O(N)$ is
the orthogonal group of $\R^N$, and let
\begin{equation}\label{X-S}
  X_{\S} := \bigl\{v\in X_k^+ \bigm| \forall s\in \S,\ \forall x \in \R^N,\
  v(s^{-1}(x)) = v(x)      \bigr\}
\end{equation}
be the set of functions invariant by the action of $\S$.
Let $\sigma : \S \to \{-1,1\}$ be a group morphism and define a second
action of $\S$ on $X$ by
\begin{equation}
  \label{eq:odd-action}
  (s \acto v)(x) := \sigma(s) \, v\bigl(s^{-1}(x)\bigr).
\end{equation}
The invariant subspace of $X_k^+ \times X_k^{\pm}$ of interest is
\begin{equation}\label{H}
  \begin{split}
    \Z : = \bigl\{z = (z_1, z_2) \in X_k^+\times X_k^{\pm} \bigm|
    \forall s\in \S,\  &z_1(s^{-1}(x)) = z_1(x) \text{ and }\\
    &\sigma(s)\, z_2(s^{-1}(x)) = z_2(x)
    \thinspace \bigr\}.
  \end{split}
\end{equation}
Then we can prove the following result:
\begin{lemma}
  The operator $T$ defined in \eqref{T} maps $\R \times
  (\X \cap \Z)$ into $\Z$.
\end{lemma}
\begin{proof}
  We will show that $T = (T_1, T_2)$ is equivariant under the action
  of $\S$, namely
 \begin{align*}
    T_1\bigl(\alpha, z_1(s^{-1}(x)),  \sigma(s)z_2 (s^{-1}(x))\bigr)
    &=T_1\bigl(\alpha, z_1(x), z_2 (x)\bigr) , \\
    \text{and}\quad
   T_2\bigl(\alpha, z_1(s^{-1}(x)),  \sigma(s)z_2 (s^{-1}(x))\bigr)
    &=\sigma(s) \, T_2\bigl(\alpha, z_1(x), z_2 (x)\bigr).
  \end{align*}
  Let $z = (z_1, z_2) \in \X$.  First, notice that,
  thanks to~\eqref{2.3-b}, the
  functions $f_1$ and $f_2$ defined in~\eqref{f1}--\eqref{f2} satisfy
  \begin{align*}
    f_1\bigl(\abs{x}, z_1(s^{-1}(x)),  \sigma(s)z_2 (s^{-1}(x))\bigr)
    &=  f_1(\abs{x}, z(x)) , \\
    \text{and}\quad
    f_2\bigl(\abs{x}, z_1(s^{-1}(x)),  \sigma(s)z_2 (s^{-1}(x))\bigr)
    &= \sigma(s) f_2(\abs{x}, z(x)).
  \end{align*}
  Second, because the Laplacian is equivariant under the action of the
  group $O(N)$, it readily follows that 
  $(-\Delta)^{-1} \bigl(\sigma(s)
  f(s(x))\bigr) = \sigma(s) \bigl((-\Delta)^{-1} f(s(x))\bigr)$ for any $\sigma$, $s\in\S$ and $f \in L^{2N/(N+2)}$.

  Putting these  observations together concludes the proof.
\end{proof}

\begin{lemma}\label{lin-invert}
  Assume $\b(\a) \ne \b_n$ for all $n \in \N$,
  with $\b_n$ be as defined in \eqref{eq:bn},
  and that the subspace of solutions in $X_\S$ to the first equation
  of~\eqref{linearization} has only the trivial solution.
  Still denote
  $T$ the operator defined in~\eqref{T} restricted to $\X \cap \Z$.
  Then the linear map $\partial_z T(\a, 0,0) : \Z \to \Z$ is invertible,
  where $\partial_z T(\a, 0,0)$ is the Fr\'echet derivative of
  $T$ with respect to $z$ at $(\a,0,0)$.
\end{lemma}
\begin{proof}
  For any $(w_1, w_2) \in X^2$, one has, see \eqref{linearization},
  \begin{equation}
    \label{eq:dzT}
    \partial_z T(\a, 0,0)
    \begin{pmatrix} w_1\\w_2 \end{pmatrix}
    = \begin{pmatrix}
      w_1-(-\Delta)^{-1}
      \Bigl(
      \frac{N+2}{N-2} U^{\frac{4}{N-2}} \, w_1 \Bigr) \\[3\jot]
      w_2 - (-\Delta)^{-1} \Bigl(\b(\a) \,
      U^{\frac{4}{N-2}} \, w_2 \Bigr)
    \end{pmatrix}
  \end{equation}
  with $\b(\a)$ as defined in \eqref{F-transversal}.
  Since $\partial_z T(\a, 0,0)$ is a compact perturbation of the
  identity (see Lemma~3.5 in~\cite{GGT}), it suffices to
  prove that $\ker\bigl(\partial_z T( \a,0,0)\bigr) = \{(0,0)\}$
  in $\Z$ whenever $\b(\a) \ne \b_n$.
  Let $(w_1, w_2) \in \Z \subseteq X_k^+ \times X_k^{\pm}$.
  Notice that
  $\partial_z T( \a,0,0)
  \left(\begin{smallmatrix} w_1 \\ w_2\end{smallmatrix}\right)
  = \left(\begin{smallmatrix} 0\\
      0 \end{smallmatrix}\right)$ if and only if $(w_1, w_2)$ is a
  solution to \eqref{linearization}. 
  By assumption we have that $w_1\equiv0$ and
  Proposition~\ref{prop-lin} says that the only solutions to the
  second equation are given
  by~\eqref{primo-caso} as we assumed $\b(\a) \ne \b_n$.
  This gives the claim.
\end{proof}

\begin{remark}
  From Lemma \ref{lin-invert} we have that, when $\b(\a) \ne \b_n$ for all
  $n$,
  \begin{equation}\label{deg}
    \deg\bigl(T(\a,\cdot), \widetilde{B}, 0\bigr)
    = \deg\bigl(\partial_z T(\a,0,0), \widetilde{B}, 0 \bigr)
    = (-1)^{m(\a)}
  \end{equation}
  where $\widetilde{B}$ is a suitable ball in $\Z$ centered
  at the origin and $m(\a)$ 
the sum
  of the algebraic multiplicities of all eigenvalues $\lambda$
  belonging to $(0,1)$ of the problem
  \begin{equation}\label{lambda}
    \begin{cases}
      \displaystyle
      -\Delta w_1 = \l \, \tfrac{N+2}{N-2} U^{\frac{4}{N-2}}  w_1
      & \text{in }\R^N, \\[1\jot]
      \displaystyle
      -\Delta w_2 = \l \, \b(\a) \, U^{\frac{4}{N-2}} w_2
      & \text{in }\R^N,\\[1\jot]
      (w_1,w_2)\in \Z.
    \end{cases}
  \end{equation}
\end{remark}

\begin{proposition}\label{p-gamma}
  Assume the same hypotheses as in Lemma \ref{lin-invert}. Let $n \in
  \N$ and $\a^*_n$ be such that $\b(\a^*_n) = \b_n$   (recall that $\b_n$
  is defined in~\eqref{eq:bn}).
  For $\e > 0$ small enough, the following holds
  \begin{equation}\label{ind-morse}
    m(\a^*_n +\e) = m(\a^*_n -\e) + \gamma(n)
  \end{equation}
  where $\gamma(n)$ is the algebraic multiplicity of the solutions to
  $-\Delta w = \b_n U^{\frac{4}{N-2}} \, w$ such that $(0,w)\in \Z$.
\end{proposition}
\begin{proof}
  As the first equation of~\eqref{lambda} does not depend on $\a$,
  its contribution is the same to the values $m(\a^*_n\pm\e)$.
  Concerning the second one, since $\b(\a)$ is a continuous increasing
  function we have get that $\b(\a^*_n +\e)\searrow\b(\a^*_n)$ and then
  the contribution of the second equation to
  $m(\a^*_n +\e)$ is given by the algebraic multiplicity of the
  eigenvalues $\lambda=\left\{\frac1{\b(\a^*_n +\e)},\dotsc,
    \frac{\b_n}{\b(\a^*_n+\e)}\right\}$. In the same way, for $\e$ small
  enough we have that  $m(\a^*_n -\e)$ is given by the algebraic
  multiplicity of the eigenvalues
  $\lambda=\left\{\frac1{\b(\a^*_n +\e)},\dotsc,
    \frac{\b_{n-1}}{\b(\a^*_n +\e)}\right\}$. This gives the claim.
\end{proof}
\begin{proposition}\label{mf1}
  Assume the same hypotheses as in Lemma \ref{lin-invert} and let us
  suppose that $\gamma(n)$ is an odd integer.
  Then the point  $(\a^*_n,
  U,U)$ is a  bifurcation point from  the curve of trivial solutions
  $(\a,U,U)$ to System~\eqref{eq:system}. Moreover the bifurcation is
  global, the Rabinowitz alternative holds, and for any sequence
  $(\a_k, u_{1,k}, u_{2,k})$ of solutions converging to
  $(\a^*_n, U,U)$, we have that
  \begin{equation*}
    \begin{cases}
      u_{1,k} = U+\frac{z_{1,k} +z_{2,k}}2 \\[1\jot]
      u_{2,k} =  U+\frac{z_{1,k} -z_{2,k}}2
    \end{cases}
  \end{equation*}
  and, up to a subsequence,
  \begin{equation}\label{mf6}
    \begin{cases}
      u_{1,k}
      = U + \e_k Z_n
      + o(\e_k) ,\\
      u_{2,k}
      = U - \e_k Z_n
      + o(\e_k) ,
      \end{cases}
    \end{equation}
    as $k\to \infty$
    where $Z_n$ is a solution to the second equation in
    \eqref{linearization} such that $(0, Z_n) \in \Z$,
    $\norm{Z_n}_X = 1$ and $\e_k = \norm{z_{2,k}}_X \to 0$.
\end{proposition}
\begin{proof}
  From \eqref{deg} and \eqref{ind-morse}, it is standard to see that
  the curve of trivial solutions for the operator
  $T : \R \times (\X \cap \Z) \to \Z$ bifurcates
  at the values $\a^*_n$ with $\b(\a^*_n) = \b_n$ for any $n$ such that
  $\gamma(n)$ is odd, see \cite[Theorem II.3.2]{kilo} and the
  bifurcation is global. The Rabinowitz alternative finally follows
  from~\cite[Theorem II.3.3]{kilo}.

  Next let us show the expansion \eqref{num0}. Let
  $(z_{1,k},z_{2,k})$ be solutions obtained by the bifurcation
  result to \eqref{eq:system} as $\a_k\to\a^*_n$ (recall that
  $(z_{1,k},z_{2,k})\rightarrow(0,0)$ in the space $X$). First we show
  that
  \begin{equation}\label{mf3}
    \frac{\norm{z_{1,k}}}{\norm{z_{2,k}}} \le C
  \end{equation}
  where $C$ is a constant independent of $k$ and
  $\norm{\cdot} = \norm{\cdot}_X$.
  First, notice that $z_{2,k} \not\equiv 0$ because, if it was,
  $z_{1,k} \in X_\S$ would satisfy
  $-\Delta z_1 = f_1(\abs{x}, z_1, 0)$ but the assumption that the
  first equation of \eqref{linearization} has only the trivial
  solution in $X_\S$ implies that this equation only has trivial
  solutions for $\a \approx \a^*_n$.  This contradicts the fact that
  $(z_{1,k}, z_{2,k})$ lies on the branch of nontrivial solutions.

  To show~\eqref{mf3}, let us argue by contradiction: let us suppose that,
  up to subsequence,
  $\frac{\|z_{1,k}\|}{\|z_{2,k}\|} \to +\infty$.  Set
  $w_{1,k}=\frac{z_{1,k}}{\|z_{1,k}\|}$,
  $w_{2,k}=\frac{z_{2,k}}{\|z_{2,k}\|}$.
  The system satisfied by $w_{1,k}$ and $w_{2,k}$ is
  \begin{subequations}
    \begin{align}
      -\Delta w_{1,k}
      & = \frac1{\|z_{1,k}\|} \biggl[F_1\Bigl(\a_k, U + \|z_{1,k}\|
        \frac{w_{1,k}+\frac{\|z_{2,k}\|}{\|z_{1,k}\|}w_{2,k}}{2},\
        U + \|z_{1,k}\|
        \frac{w_{1,k} - \frac{\|z_{2,k}\|}{\|z_{1,k}\|}w_{2,k}}{2}\Bigr)
        \nonumber\\
      &\hspace{3.2em} + F_2\Bigl(\a_k, U+\|z_{1,k}\|
        \frac{w_{1,k}+\frac{\|z_{2,k}\|}{\|z_{1,k}\|}w_{2,k}}{2},\
        U + \|z_{1,k}\|
        \frac{w_{1,k}-\frac{||z_{2,k}||}{||z_{1,k}||}w_{2,k}}{2}\Bigr)
        \nonumber\\
      &\hspace{3.2em} - 2 U^{2^*-1}\Bigr) \biggr]
        \label{mf4-1} \displaybreak[0]\\[1\jot]
      -\Delta w_{2,k}
      & = \frac1{\|z_{2,k}\|}\biggl[F_1\Bigl(\a_k, U + \|z_{2,k}\|
        \frac{\frac{\|z_{1,k}\|}{\|z_{2,k}\|}w_{1,k} + w_{2,k}}{2},\
        U + \|z_{2,k}\|
        \frac{\frac{\|z_{1,k}\|}{\|z_{2,k}\|}w_{1,k} - w_{2,k}}{2}\Bigr)
        \nonumber\\
      &\hspace{3.2em} - F_2\Bigl(\a_k, U + \|z_{2,k}\|
        \frac{\frac{\|z_{1,k}\|}{\|z_{2,k}\|}w_{1,k} + w_{2,k}}{2},\
        U + \|z_{2,k}\|
        \frac{\frac{\|z_{1,k}\|}{\|z_{2,k}\|}w_{1,k} - w_{2,k}}{2}\Bigr)\biggr]
        \label{mf4-2} \displaybreak[0]\\[1\jot]
      \multicolumn{1}{l}{\rlap{$\|w_{1,k}\| = \|w_{2,k}\| = 1$}}
      \label{mf4-3}
    \end{align}
  \end{subequations}
  Going if necessary to a subsequence, we can assume $w_{1,k} \wto
  w_1$ and $w_{2,k} \wto w_2$ in $D^{1,2}$ for some
  $(w_1, w_2) \in \Z$.  Arguing as in the
  first part of the proof of Lemma~\ref{lemma-comp}, we deduce
  that $w_{1,k} \to w_1$ and $w_{2,k} \to w_2$ in $L^{2^*}(\R^N)$ and
  in $D^{1,2}$.
  Using that $F_i(\a_k,U,U)=U^{2^*-1}$ for $i=1,2$, we can pass
  to the limit on Eq.~\eqref{mf4-1} and show that $w_1 \in X_\S$
  satisfies
  \begin{equation*}
    -\Delta w_1
    =\left[ \frac{\partial F_1}{\partial u_1}\bigl(\a^*_n, U,U\bigr)
      + \frac{\partial F_1}{\partial u_2} \bigl(\a^*_n, U,U\bigr)
      + \frac{\partial F_2}{\partial u_1} \bigl(\a^*_n, U,U\bigr)
      + \frac{\partial F_2}{\partial u_2} \bigl(\a^*_n, U,U\bigr)
    \right]\frac{w_1}{2}
  \end{equation*}
  Moreover, arguing as in the second part of the proof of
  Lemma~\ref{lemma-comp} on \eqref{mf4-1}, we can show that
  $\norm{w_{1,k} - w_1}_D \to 0$.  Thus $w_{1,k} \to w_1$ in $X$ and
  $\norm{w_1} = 1$.
  As in Section~\ref{s2}, using the properties of $F$ we have that
  $w_1\in X_\S$ satisfies
  \begin{equation*}
    -\Delta w_1 = \frac{N+2}{N-2} \, U^{\frac{4}{N-2}} \, w_1
    \quad\text{in }\R^N,
  \end{equation*}
  This is a contradiction since in $X_\S$ the previous equation
  admits only the trivial solution. So \eqref{mf3} holds.
  
  Hence, up to a subsequence, we have that
  $\frac{\|z_{1,k}\|}{\|z_{2,k}\|} \to \delta\ge0$.  Passing to the
  limit in \eqref{mf4-2}, we get that
  \begin{multline*}
    -\Delta w_2
    = \frac{\partial F_1}{\partial u_1}\bigl(\a^*_n, U,U\bigr)
    \frac{\delta w_1+w_2}{2}
    + \frac{\partial F_1}{\partial u_2}\bigl(\a^*_n, U,U\bigr)
    \frac{\delta w_1-w_2}{2}\\
    -\frac{\partial F_2}{\partial u_1}\bigl(\a^*_n, U,U\bigr)
    \frac{\delta w_1+w_2}{2}
    - \frac{\partial F_2}{\partial u_2}\bigl(\a^*_n, U,U\bigr)
    \frac{\delta w_1-w_2}{2}.    
  \end{multline*}
  and, arguing again as in the second part of the proof of
  Lemma~\ref{lemma-comp}, $w_{2,k}\to w_2$ in $X$ with
  $\norm{w_2} = 1$.
  As before, using the properties of $F$, we have that 
  $w_2$ solves
  \begin{equation*}
    -\Delta w_2 = \beta(\a) \, U^{\frac{4}{N-2}} w_2
    \quad\text{in }\R^N,
  \end{equation*}
  and hence $w_2 = Z_n$ where $Z_n$ is a solution to
  the second equation in \eqref{linearization} such that
  $(0, Z_n) \in \Z$ and $\norm{Z_n} = 1$.  Then
  $z_{2,k}= \|z_{2,k}\|(Z_n+o(1))$.  Next we show that
  \begin{equation}\label{mf5}
    z_{1,k} = o(1) \|z_{2,k}\|.
  \end{equation}
  This is clear if
  $\lim\limits_{k\rightarrow+\infty}\frac{\|z_{1,k}\|}{\|z_{2,k}\|}=0$
  since in this case
  \begin{equation}
    \norm{z_{1,k}}
    = \frac{\norm{z_{1,k}}}{\norm{z_{2,k}}} \norm{z_{2,k}}
    = o(1) \norm{z_{2,k}}.
  \end{equation}
  On the other hand, it is not possible that
  $\frac{\norm{z_{1,k}}}{\norm{z_{2,k}}}\ge D>0$ because in this case we can
  pass to the limit in~\eqref{mf4-1} and as before
  we get a contradiction. This shows \eqref{mf5}.  Coming back to the
  definition of $(u_{1,k},u_{2,k})$ we have that \eqref{mf6} holds
  with $\e_k = \|z_{2,k}\|$.
\end{proof}
Now we specify some subgroups $\S$ that satisfy the assumptions of
Lemma \ref{lin-invert}.  Observe that when $\beta(\alpha)\neq \beta_n$
the second equation in \eqref{linearization} does not possess
solutions. The first equation instead admits in $X_k^+$ the solutions
$\sum_{i=1}^N a_i\frac{x_i} {(1+|x|^2)^{N/2} }$. Then, the
assumptions of Lemma \ref{lin-invert} are satisfied if the functions
$ \frac{x_i}{(1+|x|^2)^{N/2}} $ do not belong to $X_S$.
The first example is the radial case which allows to prove Theorem
\ref{radial}. The other examples, which are provided for every
$N\ge3$, prove the existence of different nonradial solutions.

\subsection{The radial case}
\label{sec:radial-case}
Following the previous notation we let $\S=O(N)$ and 
$\sigma: \S \to \{-1,1\}$ be the group morphism such that
$\sigma(s) := 1$ for all $s \in O(N)$. 
Thus
\begin{align*}
  X_{\S}
  &= \bigl\{ v \in X \bigm| \forall x\in \R^N,\
    v(x) = v(\abs{x}) \bigr\},
  \\
  \Z\equiv\Z_{\rad}^{\pm} &= \bigl\{z\in X_k^+ \times X_k^{\pm} \bigm|
                 \forall x\in \R^N,\
                z(x) = z(\abs{x}) \bigr\}.
\end{align*}

\begin{proof}[Proof of Theorem \ref{radial}]
  To prove the bifurcation result we define the operator $T$
  in~\eqref{T} in the space
  $\Z_{\rad}^+\subseteq X_k^+\times X_k^+$ when $n$ is even
  and in the space $\Z_{\rad}^-\subseteq X_k^+\times X_k^-$
  when $n$ is odd.  Recalling the discussion at the beginning of
  Section~\ref{s2}, we have that the linearized operator
  $\partial_zT(\a,0,0)$ is invertible if and only if system
  \eqref{linearization} does not admit solutions in
  $\Z^+_{\rad}$ when $n$ is even ($\Z^-_{\rad}$ in
  case of $n$ odd).  From Proposition \ref{prop-lin} we know that the
  first equation in \eqref{linearization} does not depend on $\alpha$
  and admits the unique radial solution $W(|x|)$ which does not belong
  to $X_k^+$. The second equation in \eqref{linearization} instead
  admits solutions if and only if $\beta(\alpha)=\beta_n$ and the
  corresponding radial solution is $W_n(|x|) := W_{n,0}(r)$. Hence the
  assumption of Lemma \ref{lin-invert} are satisfied.  Moreover from
  \eqref{sol-lin} and the definition of the Jacobi polynomials we have
  that $W_n\in X_k^+$ if $n$ is even and $W_n\in X_k^-$ if $n$ is odd
  showing that $\gamma(n)=1$ for any $n$.  Further, using the
  monotonicity of $\beta(\alpha)$, the global bifurcation result
  and the Rabinowitz alternative follows from Theorem~II.3.2 and
  Theorem~II.3.3 of \cite{kilo}.  Finally
  the fact that the curve is continuously differentiable near the
  bifurcation point follows from the bifurcation result of
  Crandall-Rabinowitz for one-dimensional kernel since the operator
  $T$ is differentiable and the transversality condition holds
  in $\Z$ because
  \begin{equation*}
    \partial_{\a z}T(\a, 0,0)
    \begin{pmatrix} 0\\ W_n \end{pmatrix}
    = - \partial_\a \b(\a)
    \begin{pmatrix}
      0\\[1\jot]
      (-\Delta)^{-1} \Bigl( U^{\frac{4}{N-2}} \, W_n \Bigr)
    \end{pmatrix} ,
  \end{equation*}
  and so
  \begin{equation*}
    \biggl( \begin{pmatrix} 0\\ W_n \end{pmatrix} \biggm|
    \partial_{\a z}T(\a^*_n, 0,0)\begin{pmatrix} 0\\ W_n \end{pmatrix}
    \biggr)_{(D^{1,2})^2}
    = - \partial_\a \b(\a^*_n)
    \int_{\R^N} U^{\frac{4}{N-2}} \, W_n^2 \intd x
    \ne 0.
    \qedhere
  \end{equation*}
\end{proof}

\begin{proof}[Proof of Corollary~\ref{result-scr} ]
  It is easy to check that \eqref{F-cont}--\eqref{F-symm} are
  satisfied.  One readily computes that
   \begin{equation}
    \b(\a)=\begin{cases} 
      2 (2^* - 1 - p) \a - (2^* - 1 - 2p)&\text{in } \eqref{scr},\\[1\jot]
      2^*\a-1&\text{in } \eqref{0b},\\[1\jot]
      \frac{8}{N-2} \a - \frac{6-N}{N-2}&\text{in } \eqref{dh}.
    \end{cases} 
  \end{equation}
  and so \eqref{F-transversal} is also satisfied.  Moreover
  \eqref{eq:degen} holds if and only if $\a^* = \a^*_n$ where $\a^*_n$ is
  defined by~\eqref{3}.
  Corollary~\ref{result-scr} immediately follows.
\end{proof}

\subsection{The first nonradial case}
\label{sec:first-case}

Let $h$ be the reflection through the hyperplane $x_N=0$,
$\S_1 := \langle O(N-1), h \rangle$ be the subgroup generated by
$O(N-1)$ and $h$, and
$\sigma_1 : \S_1 \to \{-1,1\}$ be the group morphism such that
$\sigma_1(s) := 1$ if $s \in O(N-1)$ and $\sigma_1(h) := -1$
($\sigma_1$  is easily seen to be well defined because $h$ commutes
with any element of $O(N-1)$).
Thus
\begin{align*}
  X_{\S_1}
  &= \bigl\{ v \in X_k^+ \bigm| \forall x=(x',x_N)\in \R^N,\
    v(x', x_N) = v(\abs{x'}, - x_N) \bigr\},
  \\
  \Z\equiv\Z_1^{\pm} &= \bigl\{z\in X_k^+ \times X_k^{\pm} \bigm|
                 \forall x=(x',x_N)\in \R^N,\
                 \begin{aligned}[t]
                   &z_1(x',x_N) = z_1(\abs{x'}, - x_N) \text{ and }\\
                   &z_2(x',x_N) = - z_2(\abs{x'}, - x_N) \bigr\}.
                 \end{aligned}
\end{align*}
Observe that the odd symmetry helps to kill the radial solution in the
kernel of the linearized system while the even symmetries help to
avoid the solutions given by the translation invariance of the
problem. Indeed since functions in $X_{\S_1}$ are even with respect to
each $x_i$, $i = 1,\dotsc, N$ and belong to $X_k^+$ from
Proposition~\ref{prop-lin}, it is easily deduced that the solutions in
$X_{\S_1}$ of the first equation of \eqref{linearization} (see
\eqref{primo-caso}) are the trivial ones.  Thus Lemma~\ref{lin-invert}
applies and by Proposition \ref{mf1} the bifurcation result can be
proved when $\gamma(n)$ is odd.

\begin{proposition}\label{p2.12}
  With this choice of $\S = \S_1$ and $\sigma = \sigma_1$,
  we have that $\gamma(n)$ is odd if
  and only if $n = 4\ell + 1$ or $n = 4\ell + 2$ for $\ell = 0,1,\dots$
\end{proposition}
\begin{proof}
  In $\R^N$, we consider the spherical coordinates
  $(r,\varphi,\theta_1,\dots,\theta_{N-2})$ with $r=|x|\in[0,+\infty)$,
  $\varphi\in [0,2\pi]$, and $\theta_i\in [0,\pi]$ as $i=1,2,\dots,N-2$
  with
  \begin{equation}\label{eq:spherical-coord}
    \begin{cases}
      x_1 = r\cos \varphi\sin \theta_1\cdots\sin \theta_{N-2}\\
      x_2 = r\sin \varphi\sin \theta_1\cdots\sin \theta_{N-2}\\
      \hspace{0.5em}\vdots\\
      x_{N-1} = r\sin \theta_{N-2}\cos \theta_{N-3}\\
      x_N = r\cos \theta_{N-2} .
    \end{cases}
  \end{equation}
  Proposition~\ref{prop-lin} says that the solutions to
  $-\Delta w = \b_n \, U^{\frac{4}{N-2}} \, w$ 
  are,
  in radial coordinates, linear combinations of the $n+1$ functions
  \begin{equation}\label{num2}
    [0, +\infty) \times \IS^{N-1} \to \R :
    (r, \varphi,\theta_1,\dots,\theta_{N-2})
    \mapsto W_{n,k}(r)  Y_k(\varphi,\theta_1,\dots,\theta_{N-2})
  \end{equation}
  for $ k=0,\dots,n$,
  where $Y_k(\varphi,\theta_1,\dots,\theta_{N-2})$ are spherical
  harmonics with eigenvalue
  $k(k+N-2)$.   For any $k$, there is only a single (up to a scalar multiple)
  spherical harmonic which is $O(N-1)$-invariant and it is given
  by the function:
  \begin{multline}
    \label{num3}
    Y_k(\varphi,\theta_1,\dots,\theta_{N-2}) = Y_k(\theta_{N-2})
    = P_k^{(\frac{N-3}{2},\thinspace \frac{N-3}{2})}(\cos\theta_{N-2}) \\
    \quad\text{where }
    r \cos\theta_{N-2} = x_N \text{ with }
    \theta_{N-2}\in [0,\pi],
  \end{multline}
  and $P_k^{(\frac{N-3}2,\frac{N-3}2)}$ are
  the Jacobi Polynomials, see \cite{G} for example.  Then, the
  algebraic multiplicity of the solutions to $-\Delta w = \b_n \,
  U^{\frac{4}{N-2}} \, w$ that are $O(N-1)$-invariant
  is $n+1$.  By
  definition of the space $\Z_1$, the solution
  $\bigl(0, W_{n,k}(r) Y_k(\theta_{N-2}) \bigr)$
  belongs to $\Z_1$ if and only if $Y_k$ is odd with respect to $x_N$,
  that is iff
  $Y_h(\theta_{N-2}) = -Y_h(\pi-\theta_{N-2})$.  Since
  the Jacobi Polynomials are \emph{even} if $k$ is even and \emph{odd}
  if $k$ is odd,  $Y_k(\theta_{N-2})$ is odd with respect to $x_N$
  if and only if $k$
  is odd. This implies that to compute $\gamma(n)$ we only have to
  consider the odd indices $k$.

  The radial part corresponding to the index $k$ is given by 
  \begin{equation*}
    W_{n,k}(r) = \frac {r^k}{(1+r^2)^{k+\frac{N-2}2}}
    P_{n-k}^{\left(k+\frac{N-2}2, k+\frac{N-2}2\right)}
      \left(\frac{1-r^2}{1+r^2}\right).  
  \end{equation*}

  If $n=2j$, we consider the operator $T$ defined in
  $X_k^+\times X_k^-$. In this way,
  $\frac{1}{|x|^{N-2}} \cdot W_{n,k}\bigl(\frac x{|x|^2}\bigr) = -W_{n,k}(x)$
  since $n-k$ is odd for any $k$ odd.  Then
  $\gamma(n)=\sum_{k=0, \ k\text{ odd}}^n 1=j$ and it is odd if and
  only if $j=2\ell +1$, or equivalently $n=4\ell +2$.

  If, instead, $n$ is odd, then $n-k$ is even for any $k$ odd and so
  we consider the operator $T$ defined in $X_k^+\times X_k^+$. Indeed,
  in this case, $W_{n,k}(r)\in X_k^+$ for every $k$ odd and so
  $\gamma(n) = j+1$ and it is odd if and only if $j=2\ell$,
  equivalently $n = 4\ell + 1$ and this concludes the proof.
\end{proof}

\begin{proof}[Proof of Theorem \ref{teo1}]
  As explained before we are in position to apply Proposition
  \ref{mf1} using Proposition \ref{p2.12}. The expansion in
  \eqref{num0} follows again from Proposition \ref{mf1}.  Finally let
  us show that our continuum of solutions contains {\em nonradial}
  functions.  If by contradiction we have that $u_1$ and $u_2$ are
  both radial we get that $z_{2}=u_1-u_2$ is also radial. But $z_{2}$
  is odd in the last variable and so we get that $z_{2}\equiv0$. Then
  $u_1=u_2$ and by \eqref{F-1}--\eqref{F-crit} we deduce that
  $F_i(\a,u_1,u_1)= u_1^{2^*-1}$. This implies that $u_1=u_2=U$, a
  contradiction.
\end{proof}

\subsection{The general case: proof of Theorem \ref{teo3}}

Since the general case involves hard notations, for reader's
convenience we consider first the case $m = 2$ and prove Corollary
\ref{teo2}.  The general case does not involve additional difficulties
and we just will sketch it.\par
Let $h_1$ (resp.\ $h_2$) be the reflection through the hyperplane
$x_N = 0$ (resp.\ $x_{N-1} = 0$),
$\S_2 = \langle O(N-2), h_1, h_2 \rangle$ and
$\sigma_2 : \S_2 \to \{-1,1\}$ be the group morphism that satisfies
$\sigma_2(s) = 1$ whenever $s \in O(N-2)$ and
$\sigma_2(h_1) = \sigma_2(h_2) = -1$.  Thus
\begin{align*}
  X_{\S_2}
 &= \bigl\{ v\in X \bigm| \forall x=(x', x_{N-1}, x_N)\in \R^N ,\
    \begin{aligned}[t]
      &v(x', x_{N-1}, x_N) = v(|x'|, - x_{N-1}, x_N),\\
      &v(x', x_{N-1}, x_N) = v(|x'|, x_{N-1}, - x_N) \bigr\},
    \end{aligned}
  \\
  \Z\equiv\Z_2
  &= \bigl\{z\in X^+_k \times X^{\pm}_k \bigm|
    \begin{aligned}[t]
      &\forall x=(x', x_{N-1}, x_N)\in \R^N,\\
&\quad z_1(x', x_{N-1}, x_N) = z_1(|x'|, -x_{N-1}, x_N), \\
      &\quad z_1(x', x_{N-1}, x_N) = z_1(|x'|, x_{N-1}, - x_N), \\
    &\quad z_2(x', x_{N-1}, x_N) = - z_2(|x'|, - x_{N-1}, x_N),
      \text{ and}\\
      &\quad z_2(x', x_{N-1}, x_N) =- z_2(|x'|, x_{N-1}, - x_N) \bigr\},
\end{aligned}
\end{align*}
With this choice, arguing as in the previous case
we have that the only solution in $X_{\S_2}$ to the first equation
of~\eqref{linearization} is the trivial one.  As a consequence,
Proposition~\ref{mf1} applies and a bifurcation occurs
when $\gamma(n)$ is odd.

It remains to compute $\gamma(n)$. To do this we will
compute the dimension of $\Y_k^{\S_2}(\R^N)$, the space of
spherical harmonics on $\R^N$ related to the eigenvalue $k(k+N-2)$
which are invariant by the action of $\S$ induced by $\sigma$ (thus,
for $\S = \S_2$,
we select the spherical harmonics which
are invariant under the action of $O(N-2)$ and odd with respect to
$x_N$ and $x_{N-1}$).

First, let use prove the following decomposition lemma:
\begin{lemma}\label{l3.1}
  Let $\P^{\S_2}(\R^N)$ be the space of the polynomials in $N$
  variables which are invariant by the action of $O(N-2)$ and such
  that $\forall x\in \R^N,\ v(h_i(x))=-v(x)$, for $i=1,2$. Then
  \begin{equation}
    \P^{\S_2}(\R^N) = x_N x_{N-1} \R[r^2,x_{N-1}^2,x_N^2]
    \quad\text{where } r^2=x_1^2+\dots+x_{N-2}^2
  \end{equation}
  and $\R[a_1,\dots,a_k]$ denotes the
  space of polynomials in the variables $a_1, \dots,a_k$.
\end{lemma}
\begin{proof}
  \sloppy %
  The proof is similar as in Lemma 6.4 in \cite{SW}. If $p(x)$
  is a polynomial in\linebreak[3]
  $x_N x_{N-1}\R[r^2,x_{N-1}^2,x_N^2]$
  then it has an
  odd degree in $x_N$ and $x_{N-1}$ and so it satisfies
  $p(h_i(x))=-p(x)$ for $i=1,2$. Moreover it
  depends on even powers of $x_1^2+\dots+x_{N-2}^2$ and so it is
  invariant with respect to any $s\in O(N-2)$.  Thus
  $x_N x_{N-1} \R[r^2,x_{N-1}^2,x_N^2] \subseteq \P^{\S_2}(\R^N)$.

  Conversely, let $p\in \P^{\S_2}(\R^N)$. Since
  $p(h_i(x))=-p(x)$ for $i=1,2$ then each term in $p$ has to contain
  an odd power of $x_{N-1}$ and $x_N$.  We can then define the
  polynomial
  $q(x) := \frac{p(x)}{x_{N-1}x_N}$ which is even in $x_{N-1}$ and
  $x_N$.
  Now let $s\in O(N-2)$ such that $s(x_1,\dots,x_{N-2})=(r,0,\dots,0)$
  with $r^2=x_1^2+\dots+x_{N-2}^2$. Then $q$ is invariant so
  that $q(x_1,\dots,x_N) = q\bigr(s(x_1,\dots,x_{N-2}),x_{N-1},x_N\bigr)
  = q(r,0\dots,0,x_{N-1},x_N) = q(-r,0\dots,0,x_{N-1},x_N)$
  where the last equality comes from the fact that
  the map $(x_1, x_2,\dots,x_{N-2})\mapsto (-x_1, x_2,\dots,x_{N-2})$
  belongs to $O(N-2)$. Then $q$ has to be even in $r$
  and this implies that $q \in \R[r^2,x_{N-1}^2,x_N^2]$.
\end{proof}

\begin{proposition}\label{p2.14}
  With this choice of $\S = \S_2$ and $\sigma = \sigma_2$,
  $\gamma(n)$ is odd if
  and only if $n = 8\ell + 2$, $n= 8\ell + 3$, $n= 8\ell + 4$ or $n =
  8\ell + 5$ for $\ell = 0,1,\dots$
\end{proposition}
\begin{proof}
  Recall that $\Y_k(\R^N)$, the space of spherical harmonics of
  eigenvalue $k(k+N-2)$ for $-\Delta_{\IS^{N-1}}$ consists of
  harmonic homogeneous polynomials of degree~$k$.
  As stated in Proposition~5.5 of~\cite{A}, the space $\P_k$ of
  homogeneous polynomials of degree $k$ can be decomposed as a direct
  sum of $\Y_k(\R^N)$ with a subspace isomorphic to $\P_{k-2}$.
  This decomposition still holds when restricted to polynomials that
  are $O(N-2)$-invariant and odd with respect to $x_N$ and $x_{N-1}$.
  This follows
  easily using the formula (5.6) of~\cite{A}.  As a consequence,
  \begin{equation}\label{3.5}
    \dim \Y_k^{\S_2}(\R^N)
    = \dim\P_k^{\S_2}(\R^N) - \dim\P_{k-2}^{\S_2}(\R^N)
  \end{equation}
  where $\P_k^{\S_2}(\R^N)$ is the space of homogeneous
  polynomials on $\R^N$ of degree $k$ which are $O(N-2)$-invariant and
  odd with respect to $x_N$ and $x_{N-1}$.

  In view of~\eqref{3.5}, we have to compute the dimension of
  $\P_k^{\S_2}(\R^N)$ using the decomposition in Lemma~\ref{l3.1}.

   It is not difficult to show that for any $h\in\N$ we have
  $\P_{2h+1}^{\S_2}(\R^N) = \{0\}$ since  any
  polynomial in it must contain $x_{N-1}x_N$ and powers of
  $x_1^2+\dots+x_{N-2}^2$ and this is not possible if
  the degree of the polynomial is odd. So we have proved that
  $\dim\Y_{2h+1}^{\S_2}(\R^N)=0$ for any $h$ and~$N$.

  Then let us compute $\dim\Y_{2h}^{\S_2}(\R^N)$. Again from
  Lemma~\ref{l3.1}, we have that
  $\P_{k}^{\S_2}(\R^N) =
  \spanned\bigl\{ x_N^{2h+1} \, x_{N-1}^{k-2\ell-2h-1} \, r^{2\ell}
  \bigm| h=0, \dots, \frac{k-2}{2}$ and 
  $\ell = 0, \dots,\linebreak[2] \frac{k-2h-2}2 \bigr\}$ so that
  \begin{equation*}
    \dim\P_k^{\S_2}(\R^N)
    = \sum_{h=0}^{\frac{k-2}2}  \sum_{\ell=0}^{\frac{k-2h-2}{2}} 1
    = \frac{k}{4} \Bigl(\frac{k}{2}+1\Bigr)
  \end{equation*}
and using \eqref{3.5} we get for $k$ even
  \begin{equation}\label{3.6}
    \dim\Y_{k}^{\S_2}(\R^N)
    = \frac{k}{4} \Bigl(\frac{k}{2} + 1\Bigr)
    - \frac{k-2}{4} \Bigl(\frac{k-2}{2} + 1\Bigr)
    = \frac{k}{2} .
  \end{equation}
  In this case the unique spherical harmonics which contribute to the
  computation of $\gamma(n)$ are those of index $k$ even. The
  corresponding radial part is $W_{n,k}(r)$ which belongs to $X^+$ if
  $n$ is even and to $X^-$ if $n$ is odd.  Then, when $n$ is even we
  define the operator $T$ in the space $X^+\times X^+$ and we have
  that
  \begin{equation*}
    \gamma(n)
    = \sum_{k=0}^{n} \dim\Y_{k}^{\S_2}(\R^N)
    = \sum_{j=0}^{\left\lfloor\frac{n}{2}\right\rfloor} \dim\Y_{2j}^{\S_2}(\R^N)
    = \sum_{j=0}^{\left\lfloor\frac{n}{2}\right\rfloor} j
    = \frac 12 \left\lfloor\frac{n}{2}\right\rfloor
    \left(\left\lfloor\frac{n}{2}\right\rfloor + 1 \right)
  \end{equation*}
  Then $\gamma(n)$ is odd when $n=8j+2$ and $n=8j+4$.
  When $n$ is odd instead, we define the operator $T$ in the space
  $X^+\times X^-$ and we have again that
  \begin{equation*}
    \gamma(n)
    = \sum_{k=0}^{n} \dim\Y_{k}^{\S_2}(\R^N)
    = \sum_{j=0}^{\left\lfloor\frac{n}{2}\right\rfloor} \dim\Y_{2j}^{\S_2}(\R^N)
    = \sum_{j=0}^{\left\lfloor\frac{n}{2}\right\rfloor} j
    = \frac 12 \left\lfloor\frac{n}{2}\right\rfloor
    \left(\left\lfloor\frac{n}{2}\right\rfloor + 1 \right)
  \end{equation*}
  Then $\gamma(n)$ is odd when $n=8j+3$ and $n=8j+5$ concluding the proof.
\end{proof}

\begin{proof}[Proof of Corollary \ref{teo2}]
  It is the same as the one of Theorem~\ref{teo1} (using
  Proposition~\ref{p2.14}).
\end{proof}

\medskip

Now we sketch the general case of Theorem \ref{teo3}. 
Let $N\geq 3$ and $2\le m\le N-1$.
For $i = 1, \dotsc, m$ let $h_i$ be the reflection through the
hyperplane $x_{N+1-i} = 0$,
$\S_m = \langle O(N-m), h_1, \dotsc, h_m \rangle$,
and $\sigma_m : \S_m \to \{-1,1\}$ be the group morphism defined by
$\sigma_m(s) = 1$ for $s \in O(N-m)$ and $\sigma_m(h_i) = -1$.  Thus
\begin{align*}
  X_{\S_m}
  &= \bigl\{v\in X_k^+ \bigm|
    \begin{aligned}[t]
    &\forall x=(x',x_{N-m+1}, \dotsc, x_N) \in \R^{N-m}\times \R^m,\
      \forall i_1,\dotsc, i_m \in \N,\\
      &\hspace{5em}
      v(x) = v(|x'|, (-1)^{i_1} x_{N-m+1}, \dotsc, (-1)^{i_m} x_N)\},
\end{aligned}
  \\
\Z_m
&= \bigl\{z\in X_k^+ \times X_k^{\pm} \bigm|
    \begin{aligned}[t]
      &\forall x = (x', x_{N-m+1}, \dotsc, x_N)\in \R^{N-m}\times \R^m,\
      \forall i_1,\dotsc, i_m \in \N,\\
      & z_1(x) = z_1(\abs{x'}, (-1)^{i_1} x_{N-m+1}, \dotsc, (-1)^{i_m} x_N),
      \text{ and }\\
      &  z_2(x) = (-1)^{i_1+\cdots+i_m}\,
      z_2(\abs{x'}, (-1)^{i_1} x_{N-m+1}, \dotsc, (-1)^{i_m} x_N) \bigr\}.
    \end{aligned}
\end{align*}
As before we have that there is no nontrivial solution in $X_{\S_m}$
to the first equation of~\eqref{linearization}. Hence by Proposition
\ref{mf1} we only have to compute $\gamma(n)$.  Analogously to the case
$m=2$ we use the following decomposition lemma:
\begin{lemma}\label{l2.15}
  Let $\P^{S_m}(\R^N)$ be the space of the polynomials in $N$
  variables which are invariant under the action of $O(N-m)$ and such
  that $\forall x\in \R^N,\ v(h_i(x))=-v(x)$ for all $i=1,\dots,m$. Then
  \begin{equation}
    \P^{\S_m}(\R^N) = x_{N-m+1}\cdots x_{N} \, \R[r^2,x_{N-m+1}^2,\dots,x_N^2]
    \quad\text{where }
    r^2 = x_1^2 + \dots + x_{N-m}^2
  \end{equation}
  and $\R[a_1,\dots,a_k]$ denotes the space of polynomials in the
  variables $a_1, \dots,a_k$.
\end{lemma}

\begin{proposition}\label{p4.10}
  With this choice of $\S = \S_m$ and $\sigma = \sigma_m$,
  $\gamma(n)$ is odd if
  and only if
  \begin{equation*}
    \binom{m+\left\lfloor\frac{n-m}2\right\rfloor}{m}
    \text{ is an odd integer}.
  \end{equation*}
\end{proposition}

\begin{proof}
  As in the proof of Proposition \ref{p2.14} we have that
  \begin{equation}\label{3.5*}
    \dim \Y_k^{\S_m}(\R^N)
    = \dim\P_k^{\S_m}(\R^N) - \dim\P_{k-2}^{\S_m}(\R^N)
  \end{equation}
  where $\P_k^{\S_m}(\R^N)$ is the space of homogeneous
  polynomials on $\R^N$ of degree $k$ which are $O(N-m)$-invariant and
  odd with respect to $x_{N-m+1},\dots,x_{N}$.  Because of
  Lemma~\ref{l2.15}, all non-zero polynomials invariant under the action induced by $\sigma$ on $\S_m$ must have degree at least $m$ and so
  $\P_k^{\S_m}(\R^N) = \{0\}$ for $k=0,\dots,m-1$, and
  $\dim\P_{m}^{\S_m}(\R^N) = 1$. Moreover, as in the case $m=2$,
  $\P_{m+2h+1}^{\S_m}(\R^N) = \{0\}$ for any $h \in \N$.
  For $\P_{m+2h}^{\S_m}(\R^N)$, the decomposition
  in Lemma \ref{l2.15} implies that is is isomorphic to
  $\P_{h}(a_1,\dots,a_{m+1})$, the space of homogeneous
  polynomials of degree $h$ in $m+1$ variables.
  Thus $\dim \P_{m+2h}^{\S_m}(\R^N) = \dim \P_{h}(a_1,\dots,a_{m+1})
  = \binom{h+m}{m}$.  Then, using \eqref{3.5*}, we get
  \begin{equation*}
    \dim\Y_{m+2h}^{\S_m}(\R^N)
    =  \binom{h+m}{m}       - \binom{h-1+m}{m}  
    = \binom{h+m-1}{m-1},
    \qquad h \in \N.
  \end{equation*}
  This implies that $\gamma(n)=0$ for $n\leq m-1$. As when $m=2$ we get
  \begin{equation*}
    \gamma(n)
    = \sum_{h=0}^{\left\lfloor\frac {n-m}2\right\rfloor}
    \dim \Y_{m+2h}^{\S_m}(\R^N)
    = \sum_{h=0}^{\left\lfloor\frac {n-m}2\right\rfloor} \binom{h+m-1}{m-1}.
  \end{equation*}
  Now we use the so called \emph{hockey-stick} identity
  \begin{equation*}
    \sum_{i=r}^\ell \binom{i}{r} = \binom{\ell+1}{r+1}
  \end{equation*}
  which implies
  \begin{equation*}
    \gamma(n)
    = \sum_{h=0}^{\left\lfloor\frac {n-m}2\right\rfloor} \binom{h+m-1}{m-1}=
    \binom{m+\left\lfloor\frac{n-m}2\right\rfloor}{m}.
  \end{equation*}
  Finally the proof of Theorem \ref{teo3} follows as in
  Theorem~\ref{teo1}.
\end{proof}

\begin{proof}[Proof of Theorem \ref{teo3}]
  From Proposition \ref{p4.10} we have that $\gamma(n)$ is odd when
  $\binom{m+\left\lfloor\frac{n-m}2\right\rfloor}{m}$ is odd. Then the proof
  follows from Proposition \ref{mf1}.
\end{proof}

\section{Other solutions}\label{s4}

The use of other symmetry subgroups of $O(N)$ makes it possible to
find different solutions. As an example we give another choice that
generates nonradial solutions non equivalent to the previous
ones. 

For $m \ge 1$, let $R_m$ be the
rotation of angle $\frac{2\pi}{m}$ in $\varphi$,
$h_i$ the reflection with respect to
$x_{i} = 0$, $i=2,\dots,N$.  Set
$\S_m = \langle R_m, h_2, h_3,\dots,h_{N} \rangle$, and
$\sigma_m : \S_m \to \{-1,1\}$ be the group morphism defined by
$\sigma_m(R_m) = 1$, $\sigma_m(h_2) = -1$, and $\sigma_m(h_i) = 1$ for
$i = 3,\dots,N$.
(One easily checks that $\sigma_m$ is well defined using $R_m h_2
R_m = h_2$.)
Thus, using spherical coordinates, see~\eqref{eq:spherical-coord},
\begin{align*}
  X_{\S_m}
  &= \Bigl\{v\in X_k^+ \bigm|
    \forall x=(r,\varphi,\theta_1, \dotsc, \theta_{N-2}) \in \R^{N},\\
  &\hspace{5em} v( r,\varphi,\theta_1, \dotsc, \theta_{N-2}  )
    = v\left(r,2\pi-\varphi,
    \pi-\theta_1, \dotsc, \pi-\theta_{N-2}\right),\\
  &\hspace{5em} v( r,\varphi,\theta_1, \dotsc, \theta_{N-2}  )
    = v\Bigl(r,\varphi+\frac{2\pi}{m},
    \pi-\theta_1, \dotsc, \pi-\theta_{N-2}\Bigr) \Bigr\}
  \displaybreak[2]\\[1\jot]
  \Z\equiv\Z_m
  &= \Bigl\{z\in X^+_k \times X^+_k \bigm|
    \begin{aligned}[t]
      &\forall x = (r,\varphi,\theta_1, \dotsc, \theta_{N-2}) \in \R^{N},\\
      &\quad z_1(x) = z_1\left(r,2\pi-\varphi,
        \pi-\theta_1, \dotsc, \pi-\theta_{N-2}\right),\\
      &\quad z_1(x) = z_1\Bigl(r,\varphi+\frac{2\pi}{m},
      \pi-\theta_1, \dotsc, \pi-\theta_{N-2}\Bigr),\\
      &\quad z_2(x) = -z_2\left(r,2\pi-\varphi,
        \pi-\theta_1, \dotsc, \pi-\theta_{N-2}\right),\\
      &\quad z_2(x) = z_2\Bigl(r,\varphi+\frac{2\pi}{m},
      \pi-\theta_1, \dotsc, \pi-\theta_{N-2}\Bigr)        \Bigr\}.
    \end{aligned}
\end{align*}
Let us show that, for any $m\ge 2$, the first equation in
\eqref{linearization} admits only the trivial solution. By
Proposition~\ref{prop-lin}, we have that
$ w= \sum_{i=1}^N a_i\frac{x_i}{(1+|x|^2)^{N/2}} + bW$. By
\eqref{eq:spherical-coord} and the definition of $X_{\S_m}$, we get that
$a_1=a_2=0$
(using the invariance with respect to $R_m$) and
$a_3=\dots=a_{N-2}=0$ (using that $\cos\theta_i\ne\cos(\pi-\theta_i)$,
for any $i=1,\dots,\theta_{N-2}$). Finally $b=0$ since
$W\not\in X^+_k$. Thus the assumptions of Proposition~\ref{lin-invert}
are satisfied.
To apply Proposition~\ref{p-gamma}, we also need:

\begin{proposition}\label{p4}
  Let $m \ge 2$,  $n=m$, $\S := \S_n$ and $\sigma := \sigma_n$. Then
  \begin{equation}
    \gamma(n)=1.
  \end{equation}
\end{proposition}
\begin{proof}
  By Proposition \ref{prop-lin} all solutions to the second equation
  of \eqref{linearization} corresponding to $\a^*_m$ are given by
  $\sum_{k=0}^m A_k W_{m,k}(r)
  Y_k(\varphi,\theta_1,\dots,\theta_{N-2})$.
  We know from \cite{W} (see also \cite{AG} for another use of this
  expansion in bifurcation theory) that
  \begin{multline}
    \label{sf}
    Y_k(\varphi,\theta_1,\dots, \theta_{N-2}) \\
    = \!\! \sum\limits_{\substack{j=0,\dots, k\\[2pt]
        i_0\le i_1\dots\le i_{N\!-\!2}\\
        i_0=j, \ i_{N-2} =k}}
    \prod_{\ell=1}^{N-2}
    G_{i_{\ell}}^{i_{\ell-1}}(\cos\theta_\ell, \ell-1)
    \left(B_j^{i_1\dots i_{N\!-\!3}} \cos j\varphi
      + C^{i_1\dots i_{N\!-\!3}}_j \sin j\varphi\right)  \! ,
  \end{multline}
  where $G_i^0(\cdot, \ell)$ are the Gegenbauer polynomials namely,
  \begin{equation*}
    \sum\limits_{i=0}^{\infty}G_i^0(\omega,\ell) x^i
    = (1-2x\omega+x^2)^{-(1+\ell)/2} ,
  \end{equation*}
  while 
  \begin{equation*}
    G_i^k(\omega,\ell)
    =  (1-\omega^2)^{{k}/{2}} \, \frac{\diff^k}{\diff\omega^k}
    G_i^0(\omega,\ell) .
  \end{equation*}
  By definition of the space $\Z_m$, the solution
  $\bigl(0,\sum _{k=0}^m A_j W_{m,k}(r) Y_k(
  \varphi,\theta_1,\dots, \theta_{N-2} ) \bigr)$
  belongs to $\Z_m$ if and only if
  $Y_k( \varphi,\theta_1,\dots, \theta_{N-2} )$ is ${2\pi}/{m}$
  periodic in $\varphi$, changes sign
  under the transformation $\varphi \mapsto 2\pi-\varphi$,
  and is invariant under the transformations $\theta_i \mapsto
  \pi-\theta_i$.  The first two imply
  that $Y_k(\varphi,\theta_1,\dots, \theta_{N-2} )$ must not be constant
  in $\varphi$ and $k \ge j\ge m$.
  Thus solutions to the second equation in \eqref{linearization}
  with $\Z_m$-invariance are
  multiple of $W_{m,m}(r)Y_m(\theta)$.
  Moreover, the unique nonzero coefficient in \eqref{sf} is
  $C_m^{m\dots m}$.
  
  Because $G_i^0(\cdot, \ell)$ is a polynomial of degree $\ell$,
  $G_m^m(\omega, \ell)$ is a constant multiple of $(1 - \omega^2)^{m/2}$.
  A straightforward computation shows that
  \begin{equation*}
    Y_m( \varphi,\theta_1, \dots, \theta_{N-2})
    = (\sin\theta_{N-2})^m \cdots (\sin\theta_1)^m \sin(m\varphi)
    = \Im(x_1+ \mathbf{i}\thinspace x_2)^m.  
  \end{equation*}
  Observe that $Y_m( \varphi,\theta_1, \dots, \theta_{N-2})$ is invariant with
  respect to the reflection $h_i$ for $i=3,\dots,N$, with respect to
  the rotation $R_m$ and it is odd in $\varphi$ so that
  $(0,W_{m,m}(r)Y_m(\theta))$ belongs to $\mathcal{Z}_m$.  Recalling
  that
  $W_{m,m}(r) = \smash{\frac{r^m}{\left(1+r^2\right)^{m+\frac{N-2}2}}}
  \in X^+_k$,
  we have that $\gamma(m)=1$.\linebreak[0]
\end{proof}

\begin{proof}[Proof of Theorem \ref{teo4}]
  From the previous discussion we have that the assumption of Lemma
  \ref{lin-invert} are satisfied.  Then the proof follows as in the
  case of Theorem \ref{radial} since we have a one dimensional kernel.
\end{proof}

\end{document}